\documentclass[a4paper,reqno,10pt]{amsart}
\makeatletter
\newcommand*{\rom}[1]{\expandafter\@slowromancap\romannumeral #1@}
\makeatother
\usepackage{epsf,graphicx,epsfig}
\usepackage{amscd}
\usepackage{amsmath,latexsym,amssymb,amsthm}
\usepackage{fdsymbol}
\usepackage[nospace,noadjust]{cite}
\usepackage{textcomp}
\usepackage{setspace,cite}
\usepackage{mathrsfs,amsmath}
\usepackage{chemarrow}
\usepackage{lscape,fancyhdr,fancybox}
\usepackage{stmaryrd}
\usepackage[all,cmtip]{xy}
\usepackage{tikz-cd}
\usepackage{cancel}
\usepackage[usestackEOL]{stackengine}
\usetikzlibrary{shapes,arrows,decorations.markings}
\usepackage{graphicx}

\usepackage{color}
\usepackage{bbm, dsfont}
\usepackage{xcolor}
\usepackage{url}
\usepackage{enumerate}
\usepackage[mathscr]{euscript}

  \theoremstyle{plain}
\swapnumbers
    \newtheorem{thm}{Theorem}[section]
    \newtheorem{proposition}[thm]{Proposition}
     
   \newtheorem{lemma}[thm]{Lemma}
    \newtheorem{corollary}[thm]{Corollary}
    
    \newtheorem{subsec}[thm]{}
\theoremstyle{definition}
    \newtheorem{definition}[thm]{Definition}
        \newtheorem{remark}[thm]{Remark}
    \newtheorem{exam}[thm]{Example}

\theoremstyle{remark}

\title{}
\author{}
\date{}
\usepackage{amssymb}

\usepackage{hyperref}
\hypersetup{
colorlinks,
citecolor=blue,
filecolor=black,
linkcolor=blue,
urlcolor=black
}

\usepackage[left=25mm,right=25mm,top=25mm,bottom=25mm,paper=a4paper]{geometry}

\begin{document}

\title[Cup product, Fr\"{o}licher-Nijenhuis bracket and the derived bracket]{Cup product, Fr\"{o}licher-Nijenhuis bracket and the derived bracket associated to Hom-Lie algebras}

\author{Anusuiya Baishya}
\address{Department of Mathematics,
Indian Institute of Technology, Kharagpur 721302, West Bengal, India.}
\email{anusuiyabaishya530@gmail.com}

\author{Apurba Das}
\address{Department of Mathematics,
Indian Institute of Technology, Kharagpur 721302, West Bengal, India.}
\email{apurbadas348@gmail.com, apurbadas348@maths.iitkgp.ac.in}


\begin{abstract}
In this paper, we introduce some new graded Lie algebras associated with a Hom-Lie algebra. At first, we define the cup product bracket and its application to the deformation theory of Hom-Lie algebra morphisms. We observe an action of the well-known Hom-analogue of the Nijenhuis-Richardson graded Lie algebra on the cup product graded Lie algebra. Using the corresponding semidirect product, we define the Fr\"{o}licher-Nijenhuis bracket and study its application to Nijenhuis operators. We show that the Nijenhuis-Richardson graded Lie algebra and the Fr\"{o}licher-Nijenhuis algebra constitute a matched pair of graded Lie algebras. Finally, we define another graded Lie bracket, called the derived bracket that is useful to study Rota-Baxter operators on Hom-Lie algebras.
\end{abstract}

\maketitle


\medskip

\medskip
\begin{center}

    

    \end{center}

\noindent

\thispagestyle{empty}

\tableofcontents

\vspace{0.2cm}

\section{Introduction}\label{sec1}
 The classical Nijenhuis-Richardson bracket plays a prominent role in the Lie algebra theory as it determines Lie algebra structures on a given vector space as Maurer-Cartan elements. It has been shown in \cite{nij-ric-1,nij-ric-2} that the Nijenhuis-Richardson bracket is useful to show that the obstruction for extending a finite order deformation of a Lie algebra is a third cocycle. Further, given a Lie algebra $\mathfrak{g}$, the Chevalley-Eilenberg differential of the Lie algebra $\mathfrak{g}$ is a graded derivation for the Nijenhuis-Richardson bracket. Another well-known operation on the Chevalley-Eilenberg cochain complex of a Lie algebra is the cup product operation \cite{nij-ric-mor}. This cup product is a generalization of the Gerstenhaber's cup product \cite{gers-ring} and it plays an important role in the study of deformations of Lie algebra homomorphisms. The Chevalley-Eilenberg differential of the Lie algebra is also a graded derivation for the cup product operation. It is further well-known that the induced cup product operation at the level of Chevalley-Eilenberg cohomology is trivial. Although the Nijenhuis-Richardson bracket and the cup product operation both provide rich structures, the relationship between these two operations is not very well-known. See \cite{yang,nij-ric-1,nij-ric-mor} for some results.

\medskip

On the other hand, the notion of Hom-Lie algebras was introduced by Hartwig, Larsson and Silvestrov \cite{hls} in the $q$-deformations of Witt and Virasoro Lie algebras. In a Hom-Lie algebra, the skew-symmetric bilinear bracket does not generally satisfy the Jacobi identity. However, it satisfies a generalization of the Jacobi identity twisted by a linear map, called the Hom-Jacobi identity. When the twist map is identity, one recovers a classical Lie algebra. It has been observed in \cite{hls} that some $q$-deformations of Witt and Virasoro algebras have the structure of a Hom-Lie algebra. A general study and various constructions of Hom-Lie algebras were given in \cite{ls1,ls2,am1,jin-li}. In the last twenty years, Hom-Lie algebras have been paid much attention because of their close relationship with discrete and deformed vector fields \cite{ls1,ls2,ls3}. Among others, homology theory \cite{yau}, cohomology and deformation theory \cite{amm-ej-makh,makh-sil-2}, representations \cite{sheng} and enveloping algebras of Hom-Lie algebras \cite{gengoux} are extensively studied. In \cite{benayadi}, the authors also considered quadratic Hom-Lie algebras that allow one to study Manin triples of Hom-Lie algebras and hence Hom-Lie bialgebras \cite{sheng-bai}.

\medskip

In \cite{amm-ej-makh} the authors considered a generalization of the classical Nijenhuis-Richardson bracket on the space of all skew-symmetric multilinear maps on a given vector space that commute with a fixed twist map $\alpha$. This bracket, which is also referred as the Nijenhuis-Richardson bracket satisfies the graded Jacobi identity. The Maurer-Cartan elements of this graded Lie bracket correspond to multiplicative Hom-Lie algebra structures on the given vector space with the twist map $\alpha$. This Nijenhuis-Richardson bracket plays a crucial role in cohomology and deformations of Hom-Lie algebras in the same way the classical Nijenhuis-Richardson bracket plays for Lie algebras.

\medskip

In the present paper, we initiate our study by introducing the cup product operation $[~,~]_\mathsf{C}$ (also called the cup product bracket) on the graded space $C^\bullet_\mathrm{Hom} (\mathfrak{g}, \mathfrak{g})$ of cochains of a Hom-Lie algebra $\mathfrak{g}$. We show that it defines a graded Lie algebra structure. This cup product bracket endowed with a suitable differential makes a differential graded Lie algebra that controls the deformations of Hom-Lie algebra morphisms. We find some other descriptions of the cup product bracket that are also useful in the course of our study. Among others, we show that the induced cup product operation at the level of the cohomology is trivial. This generalizes the similar result from the Lie algebra context. Next, given a Hom-Lie algebra, we show that the corresponding Nijenhuis-Richardson bracket acts on the cup product Lie bracket by graded derivations.  Thus, we obtain the semidirect product graded Lie algebra structure on their direct sum. As a byproduct, we define a bracket $[~,~]_\mathsf{FN} : C^\bullet_\mathrm{Hom} (\mathfrak{g}, \mathfrak{g}) \times C^\bullet_\mathrm{Hom} (\mathfrak{g}, \mathfrak{g}) \rightarrow C^\bullet_\mathrm{Hom} (\mathfrak{g}, \mathfrak{g})$ by
\begin{align*}
    [P, Q]_\mathsf{FN} := [P, Q]_\mathsf{C} + (-1)^m ~ i_{(\delta_\mathrm{Hom} P)}Q  - (-1)^{(m+1) n} ~ i_{(\delta_\mathrm{Hom} Q)} P, 
\end{align*}
for $P \in C^m_\mathrm{Hom} (\mathfrak{g}, \mathfrak{g})$ and $Q \in C^n_\mathrm{Hom} (\mathfrak{g}, \mathfrak{g})$. Here $\delta_\mathrm{Hom}$ is the Hom-Lie algebra differential. This bracket turns out to be a graded Lie bracket, called the Fr\"{o}licher-Nijenhuis bracket and the graded Lie algebra $( C^\bullet_\mathrm{Hom} ( \mathfrak{g}, \mathfrak{g}), [~,~]_\mathsf{FN})$ is called the Fr\"{o}licher-Nijenhuis algebra associated with the Hom-Lie algebra $\mathfrak{g}$. 
Like the Maurer-Cartan characterization of Nijenhuis operators in classical geometry \cite{fro-nij-1,fro-nij-2}, here we observe that a linear map $N: \mathfrak{g} \rightarrow \mathfrak{g}$ satisfying $\alpha \circ N = N \circ \alpha$ is a Nijenhuis operator on the Hom-Lie algebra if and only if $[N, N]_\mathsf{FN} = 0$.
It follows that the Fr\"{o}licher-Nijenhuis bracket characterizes Nijenhuis operators as Maurer-Cartan elements. 
We also observed that the Nijenhuis-Richardson graded Lie algebra and the Fr\"{o}licher-Nijenhuis algebra forms a matched pair of graded Lie algebras.

\medskip

In the last few years, Rota-Baxter operators on Lie algebras have been extensively studied due to their connections with pre-Lie algebras, post-Lie algebras, classical Yang-Baxter equation and integrable systems \cite{bai-guo-ni,kuper}. The Maurer-Cartan characterization and cohomology of Rota-Baxter operators of weight $0$ and nonzero weight are respectively studied in \cite{tang} and \cite{das-weighted}. The weight $0$ case was generalized to the context of Hom-Lie algebras in \cite{mishra-naolekar}. Namely, the authors first construct a graded Lie algebra (generalizing the graded Lie algebra obtained in \cite{tang}) associated with a Hom-Lie algebra and show that its Maurer-Cartan elements are Rota-Baxter operators of weight $0$. Since their bracket is obtained by Voronov's derived bracket construction, we call it a derived bracket. As a consequence of our semidirect product graded Lie algebra constructed earlier and by another byproduct, we can define a new bracket $[~,~]_\mathsf{D}: C^\bullet_\mathrm{Hom} (\mathfrak{g}, \mathfrak{g}) \times C^\bullet_\mathrm{Hom} (\mathfrak{g}, \mathfrak{g}) \rightarrow C^\bullet_\mathrm{Hom} (\mathfrak{g}, \mathfrak{g})$ by
\begin{align}
[P,Q]_{\mathsf{D}} : = [P, Q ]_{\mathsf{C}} +i_{\theta P} Q - (-1)^{mn} i_{\theta Q} P,
\end{align}
for $P \in C^m_\mathrm{Hom} (\mathfrak{g}, \mathfrak{g})$ and $Q \in C^n_\mathrm{Hom} (\mathfrak{g}, \mathfrak{g})$. Here $\theta$ is some suitable map in the context. This bracket is also a graded Lie bracket. An explicit description of this bracket shows that it coincides with the derived bracket considered in \cite{mishra-naolekar} (hence in the Lie algebra context, it coincides with the bracket given in \cite{tang}). In other words, the derived bracket can be expressed in terms of the cup product bracket and the contraction operator $i$.
For this reason, we refer $[~,~]_\mathsf{D}$ as the derived bracket and the graded Lie algebra $(C^\bullet_\mathrm{Hom} (\mathfrak{g}, \mathfrak{g}), [~,~]_\mathsf{D})$ is called the derived algebra associated with the Hom-Lie algebra $\mathfrak{g}$. Our description of the bracket is useful to connect it with the Nijenhuis-Richardson algebra by a graded Lie algebra morphism. Given any scalar $\lambda \in {\bf k}$, we also define a differential $d_\lambda$  that makes the triple $( C^\bullet_\mathrm{Hom} (\mathfrak{g}, \mathfrak{g}) , [~,~]_\mathsf{D}, d_\lambda)$ into a differential graded Lie algebra. We show that a linear map $R: \mathfrak{g} \rightarrow \mathfrak{g}$ satisfying $\alpha \circ R = R \circ \alpha$ is a Rota-Baxter operator of weight $\lambda$ on the Hom-Lie algebra $\mathfrak{g}$ if and only if 
\begin{align*}
    d_\lambda R + \frac{1}{2} [R,R]_\mathsf{D} = 0.
\end{align*}
In particular, $R$ is a Rota-Baxter operator of weight $0$ if and only if $[R,R]_\mathsf{D} = 0$. Therefore, we obtain Maurer-Cartan characterization of Rota-Baxter operators. Finally, we generalize the above differential graded Lie algebra to study relative Rota-Baxter operators of any weight $\lambda$. Among others, we define the cohomology of a relative Rota-Baxter operator of weight $\lambda$ and interpret it as the cohomology of a suitable Hom-Lie algebra.


\medskip

\noindent {\bf Organization of the paper:} The paper is organized as follows. In Section \ref{sec2}, we recall Hom-Lie algebras, their cochains and the (Hom-analogue of the classical) Nijenhuis-Richardson bracket. In Section \ref{sec3}, we first define the cup product bracket and a suitable differential that makes a differential graded Lie algebra. In terms of this differential graded Lie algebra, we study cohomology and deformations of a Hom-Lie algebra morphism. We begin Section \ref{sec4} by defining an action of the Nijenhuis-Richardson graded Lie algebra on the cup product algebra. Then we define the Fr\"{o}licher-Nijenhuis bracket and study Nijenhuis operators. We also show that the Nijenhuis-Richardson algebra and the Fr\"{o}licher-Nijenhuis algebra forms a matched pair of graded Lie algebras. Finally, in Section \ref{sec5}, we define the differential graded Lie algebra $( C^\bullet_\mathrm{Hom} (\mathfrak{g}, \mathfrak{g}) , [~,~]_\mathsf{D}, d_\lambda)$ and a generalization of it to study (relative) Rota-Baxter operators of any weight $\lambda$.

\medskip

\noindent {\bf Notations:} All vector spaces, (multi)linear maps, tensor and wedge products are over a field ${\bf k}$ of characteristic $0$ unless specified otherwise. Let $(\mathcal{L} = \oplus_{i \in \mathbb{Z}} \mathcal{L}^i, [~,~], d)$ be a differential graded Lie algebra. A degree $1$ element $s \in \mathcal{L}^1$ is called a {\em Maurer-Cartan element} of the differential graded Lie algebra $\mathcal{L}$ if it satisfies
\begin{align*}
    d s + \frac{1}{2} [s, s] = 0.
\end{align*}
A permutation $\sigma \in S_{n_1 + \cdots + n_k}$ is called a $(n_1 , \ldots, n_k )$-shuffle if
\begin{align*}
    \sigma (1) < \cdots < \sigma (n_1), ~ \sigma (n_1 + 1) < \cdots < \sigma (n_1 + n_2),~ \ldots, \sigma (n_1 + \cdots + n_{k-1} + 1) < \cdots < \sigma (n_1 + \cdots + n_k).
\end{align*}
We denote the set of all $(n_1, \ldots, n_k)$-shuffles by $Sh (n_1, \ldots, n_k).$

\medskip

\section{Hom-Lie algebras and the Nijenhuis-Richardson bracket}\label{sec2}
In this section, we first recall Hom-Lie algebras and their associated cochain complex. Then the Hom-analogue of the Nijenhuis-Richardson bracket is also recalled whose Maurer-Cartan elements are precisely multiplicative Hom-Lie algebra structures. Our main references are \cite{amm-ej-makh,hls,makh-sil-2,sheng}.

\begin{definition}
    A {\bf Hom-Lie algebra} is a triple $(\mathfrak{g}, [~,~], \alpha)$ consisting of a vector space $\mathfrak{g}$ endowed with a skew-symmetric bilinear map (called the Hom-Lie bracket) $[~,~]: \mathfrak{g} \times \mathfrak{g} \rightarrow \mathfrak{g} $ and a linear map (called the twist map) $\alpha: \mathfrak{g} \rightarrow \mathfrak{g}$ that satisfy
\begin{align}\label{hj-iden}
 [\alpha (x) , [ y, z]] + [\alpha (y), [z,x]] + [\alpha (z), [x, y]] = 0,~~ \text{ for } x, y, z \in \mathfrak{g}.
\end{align} 
\end{definition}

The identity (\ref{hj-iden}) is called the {\em Hom-Jacobi identity}. A Hom-Lie algebra $(\mathfrak{g}, [~,~], \alpha)$ with $\alpha = \mathrm{id}_\mathfrak{g}$ is nothing but a Lie algebra. Thus, the class of Lie algebras is a subclass of the class of Hom-Lie algebras.


\medskip

Let $(\mathfrak{g}, [~,~]_\mathfrak{g}, \alpha)$ and $(\mathfrak{h}, [~,~]_\mathfrak{h}, \beta)$ be two Hom-Lie algebras. A {\bf morphism} of Hom-Lie algebras from $\mathfrak{g}$ to $\mathfrak{h}$ is a linear map $\varphi : \mathfrak{g} \rightarrow \mathfrak{h}$ satisfying $\beta \circ  \varphi = \varphi \circ \alpha $ and $\varphi ([x, y]_\mathfrak{g}) = [\varphi (x) , \varphi (y)]_\mathfrak{h}$, for $x,y \in \mathfrak{g}$.

\begin{exam}
Let $(\mathfrak{g}, [~,~])$ be a Lie algebra and $\alpha : \mathfrak{g} \rightarrow \mathfrak{g}$ be a Lie algebra homomorphism. Then the triple $(\mathfrak{g}, \alpha \circ [~,~], \alpha)$ is a Hom-Lie algebra.
\end{exam}


\begin{exam}
    Let $\mathfrak{g}$ be a $3$-dimensional vector space with a basis $\{ e_1, e_2, e_3 \}$. Define a bilinear skew-symmetric bracket $[~,~] : \mathfrak{g} \times \mathfrak{g} \rightarrow \mathfrak{g}$ and a linear map $\alpha : \mathfrak{g} \rightarrow \mathfrak{g}$ by 
    \begin{align*}
        [e_1, e_2] =~& a e_1 + b e_3, \quad [e_1, e_3] = c e_2, \quad [e_2, e_3] = de_1 + 2a e_3, \\
        &\alpha (e_1) = e_1, \quad \alpha (e_2) = 2 e_2, \quad \alpha (e_3) = 2 e_3,
    \end{align*}
    where $a, b , c , d \in {\bf k}$ are fixed scalars.
    Then $(\mathfrak{g}, [~,~], \alpha)$ is a Hom-Lie algebra.
\end{exam}

\begin{exam}
Let $(A, \mu, \alpha)$ be a Hom-associative algebra \cite{makh-sil-2}. That is, $A$ is a vector space, $\mu : A \times A \rightarrow A, ~ (a , b) \mapsto a \cdot b$ is a bilinear map and $\alpha : A \rightarrow A$ is a linear map satisfying
\begin{align*}
\alpha (a) \cdot (b \cdot c) = (a \cdot b) \cdot \alpha (c), ~ \text{ for } a, b, c \in A.
\end{align*}
If $(A, \mu, \alpha)$ is a Hom-associative algebra, then the triple $(A, [~,~], \alpha)$ is a Hom-Lie algebra, where $[a, b ] := a \cdot b - b \cdot a$, for $a, b \in A$. This is called the commutator Hom-Lie algebra of the Hom-associative algebra $(A, \mu, \alpha).$
\end{exam}

\begin{exam}\label{jack} (Jackson $\mathfrak{sl}_2$) Let $\mathfrak{g}$ be a $3$-dimensional vector space with a basis $\{ e, h, f \}$. For each parameter $q \in {\bf k}$, we define a bilinear skew-symmetric bracket $[~,~] : \mathfrak{g} \times \mathfrak{g} \rightarrow \mathfrak{g}$ and a linear map $\alpha : \mathfrak{g} \rightarrow \mathfrak{g}$ by 
\begin{align*}
    [h, e ] &= 2 e,  \quad [e, f ] = \frac{1+q}{2} h , \quad [h, f ] = - 2 q f, \\
    &\alpha (e) = q e , \quad \alpha (h) = q h, \quad \alpha (f) = q^2 f.
\end{align*}
Then $(\mathfrak{g}, [~,~], \alpha)$ is a Hom-Lie algebra, called the Jackson $\mathfrak{sl}_2$.

    Note that Jackson $\mathfrak{sl}_2$ is a $q$-deformation of the classical $\mathfrak{sl}_2$. In particular, if $q= 1$, we recover the classical $\mathfrak{sl}_2$.
\end{exam}

A Hom-Lie algebra $(\mathfrak{g}, [~,~], \alpha)$ is said to be {\bf multiplicative} if
\begin{align*}
    \alpha ([x, y]) = [\alpha (x), \alpha (y)], \text{ for all } x, y \in \mathfrak{g}.
\end{align*}
A morphism of multiplicative Hom-Lie algebras from $(\mathfrak{g}, [~,~]_\mathfrak{g}, \alpha)$ to $(\mathfrak{h}, [~,~]_\mathfrak{h}, \beta)$ is a linear map $\varphi : \mathfrak{g} \rightarrow \mathfrak{h}$ that is a Hom-Lie algebra morphism (i.e. $\varphi ([x, y]_\mathfrak{g}) = [\varphi (x), \varphi(y) ]_\mathfrak{h}$, for $x, y \in \mathfrak{g}$) and satisfying $\beta \circ \varphi = \varphi \circ \alpha$.

\medskip

\noindent \underline{\bf Note :} In this paper, we are mainly interested in multiplicative Hom-Lie algebras and our results are valid only for them. Thus, from now onward, by a Hom-Lie algebra (resp. a morphism of Hom-Lie algebras), we shall always assume a multiplicative Hom-Lie algebra (resp. a morphism of multiplicative Hom-Lie algebras) unless specified otherwise.

\begin{definition} Let $(\mathfrak{g}, [~,~], \alpha)$ be a Hom-Lie algebra. A {\bf representation} of $\mathfrak{g}$ consists of a vector space $V$ together with a bilinear map $\diamond : \mathfrak{g} \times V \rightarrow V$ and a linear map  $\beta : V \rightarrow V$ satisfying $\beta (x \diamond v) = \alpha (x) \diamond \beta (v)$ and
\begin{align*}
[x, y]  \diamond \beta (v) = \alpha (x) \diamond ( y \diamond v) - \alpha (y) \diamond (x \diamond v),~ \text{ for } x, y \in \mathfrak{g}, v \in V.
\end{align*}
\end{definition}

A representation as above may be denoted by the triple  $(V, \diamond, \beta)$ or simply by $V$ when no confusion arises. It follows from the above definition that any Hom-Lie algebra $(\mathfrak{g}, [~,~], \alpha)$ can be regarded as a representation $(\mathfrak{g}, \diamond = [~,~], \alpha)$ of itself, called the {\em adjoint representation}.

Next, we recall the cohomology of a Hom-Lie algebra $(\mathfrak{g}, [~,~], \alpha)$ with coefficients in a given representation $(V, \diamond, \beta)$. For each $n \geq 0$, the $n$-th cochain group $C^n_{\mathrm{Hom}} (\mathfrak{g}, V)$ is given by
\begin{align*}
C^0_{\mathrm{Hom}} (\mathfrak{g}, V) = \{ v \in V ~|~ \beta (v) = v \} ~~~ \text{ and } ~~~ C^{n \geq 1}_{\mathrm{Hom}} (\mathfrak{g}, V) = \{ f : \wedge^n \mathfrak{g} \rightarrow V ~|~ \beta \circ f = f \circ \alpha^{\wedge^n} \}.
\end{align*}
The differential $\delta_{\mathrm{Hom}} : C^n_{\mathrm{Hom}} (\mathfrak{g}, V) \rightarrow C^{n+1}_{\mathrm{Hom}} (\mathfrak{g}, V)$, for $n \geq 0$, is given by
\begin{align}
\delta_{\mathrm{Hom}} (v)(x) =~& x \diamond v,\\
(\delta_{\mathrm{Hom}} f)(x_1, \ldots, x_{n+1}) =~& \sum_{i=1}^{n+1} (-1)^{i+1} ~  \alpha^{n-1} (x_i) \diamond f (x_1, \ldots, \widehat{x_i}, \ldots, x_{n+1}) \label{hom-lie-coho}\\
~&+ \sum_{1 \leq i < j \leq n+1} (-1)^{i+j} ~ f ( [x_i, x_j], \alpha (x_1), \ldots, \widehat{\alpha (x_i)}, \ldots, \widehat{\alpha (x_j)}, \ldots, \alpha (x_{n+1}) ), \nonumber
\end{align}
for $v \in C^0_\mathrm{Hom} (\mathfrak{g}, V)$, $f \in C^{n \geq 1}_{\mathrm{Hom}} (\mathfrak{g}, V)$ and $x, x_1, \ldots, x_{n+1}\in \mathfrak{g}$.
The corresponding cohomology groups are called the {\em cohomology} of the Hom-Lie algebra $(\mathfrak{g}, [~,~], \alpha)$ with coefficients in $(V, \diamond, \beta)$, and they are denoted by $H^\bullet_{\mathrm{Hom}} (\mathfrak{g}, V)$. 

When $\mathfrak{g}$ is a Lie algebra and $V$ is a Lie algebra representation (that is, when $\alpha = \mathrm{id}_\mathfrak{g}$ and $\beta = \mathrm{id}_V$), one recovers the classical Chevalley-Eilenberg cohomology of the Lie algebra.


\medskip

\noindent {\bf Nijenhuis-Richardson bracket.}
Let $\mathfrak{g}$ be a vector space and $\alpha: \mathfrak{g} \rightarrow \mathfrak{g}$ be a linear map. (Note that $\mathfrak{g}$ need not be a Hom-Lie algebra.) For each $n \geq 1$, we define 
\begin{align*}
     C^{n}_{\mathrm{Hom}} (\mathfrak{g}, \mathfrak{g}) = \{ f : \wedge^n \mathfrak{g} \rightarrow \mathfrak{g} ~|~ \alpha \circ f = f \circ \alpha^{\wedge^n} \}.
\end{align*}
It has been shown in \cite{amm-ej-makh} that the graded space $C^\bullet_{\mathrm{Hom}} (\mathfrak{g}, \mathfrak{g}) := \bigoplus_{n \geq 1} C^n_{\mathrm{Hom}} (\mathfrak{g}, \mathfrak{g})$ carries a degree $-1$ graded Lie bracket (also called the {\em Nijenhuis-Richardson bracket} in the present context) that generalizes the classical Nijenhuis-Richardson bracket  given by
\begin{align*}
    [P, Q]_{\mathsf{NR}} := i_P Q - (-1)^{(m-1)(n-1)} i_Q P, ~~~ \text{ where }
\end{align*}
\begin{align}\label{contraction-map}
    (i_P  Q) (x_1, \ldots, x_{m+n-1}) = \sum_{\sigma \in Sh (m, n-1)} (-1)^\sigma ~Q  \big( P (x_{\sigma (1)}, \ldots, x_{\sigma (m)}), \alpha^{m-1} x_{\sigma (m+1)}, \ldots, \alpha^{m-1}x_{\sigma (m+n-1)} \big),
\end{align}
for $P \in C^m_{\mathrm{Hom}} (\mathfrak{g}, \mathfrak{g})$, $Q \in C^n_{\mathrm{Hom}} (\mathfrak{g}, \mathfrak{g})$ and $x_1, \ldots, x_{m+n-1} \in \mathfrak{g}$. In other words, the shifted graded space $C^{\bullet + 1}_{\mathrm{Hom}} (\mathfrak{g}, \mathfrak{g}) = C^\bullet_{\mathrm{Hom}} (\mathfrak{g}, \mathfrak{g})[1]$ is a graded Lie algebra with the above Nijenhuis-Richardson bracket. Note that an element $\mu \in C^2_\mathrm{Hom} (\mathfrak{g}, \mathfrak{g})$ corresponds to a bilinear skew-symmetric multiplicative bracket $[~,~] : \mathfrak{g} \times \mathfrak{g} \rightarrow \mathfrak{g}$ given by 
\begin{align*}
[x, y]:= \mu (x, y), \text{ for } x, y \in \mathfrak{g}.
\end{align*}
With the above notations, $\mu$ is a Maurer-Cartan element of the graded Lie algebra $( C^{\bullet + 1}_{\mathrm{Hom}} (\mathfrak{g}, \mathfrak{g}) , [~,~]_\mathsf{NR})$ if and only if $(\mathfrak{g}, [~,~], \alpha)$ is a Hom-Lie algebra. Thus, Hom-Lie brackets are characterized by Maurer-Cartan elements.

Let $(\mathfrak{g}, [~,~], \alpha)$ be a Hom-Lie algebra. Note that the 
differential (\ref{hom-lie-coho}) of the Hom-Lie algebra $(\mathfrak{g}, [~,~], \alpha)$ with coefficients in the adjoint representation can simply be written as 
\begin{align*}
    \delta_{\mathrm{Hom}} (f) = -[\mu, f]_{\mathsf{NR}}, ~ \text{ for } f \in C^n_{\mathrm{Hom}} (\mathfrak{g}, \mathfrak{g}),
\end{align*}
where $\mu \in C^2_{\mathrm{Hom}} (\mathfrak{g}, \mathfrak{g})$ is the Maurer-Cartan element corresponding to the Hom-Lie bracket $[~,~]$. This description of the differential also says that $\delta_\mathrm{Hom}$ is a graded derivation for the graded Lie algebra $(C^{\bullet + 1 }_\mathrm{Hom} (\mathfrak{g}, \mathfrak{g}), [~,~]_\mathsf{NR})$.


\section{Cup product bracket for Hom-Lie algebras}\label{sec3}
In this section, we first define the cup product bracket in the context of Hom-Lie algebras and show that it defines a graded Lie algebra structure. This graded Lie algebra together with a suitable differential becomes a differential graded Lie algebra whose Maurer-Cartan elements correspond to Hom-Lie algebra morphisms. As applications, we study cohomology and deformations of a Hom-Lie algebra morphism. In the end, we provide some new descriptions of the cup product bracket and show that it induces a trivial operation at the level of cohomology.

Let $(\mathfrak{g}, [~,~], \alpha)$ be a Hom-Lie algebra. For each $m, n \geq 1$, we define a bracket
\begin{align*}
    [~,~]_\mathsf{C} : C^m_\mathrm{Hom} (\mathfrak{g}, \mathfrak{g}) \times C^n_\mathrm{Hom} (\mathfrak{g}, \mathfrak{g}) \rightarrow C^{m+n}_\mathrm{Hom} (\mathfrak{g}, \mathfrak{g}) 
\end{align*}
by
\begin{align}\label{cupp}
    [P, Q]_\mathsf{C} (x_1, \ldots, x_{m+n} ) := \sum_{\sigma \in Sh (m,n)} (-1)^\sigma ~[\alpha^{n-1} P (x_{\sigma (1)}, \ldots, x_{\sigma (m)}), \alpha^{m-1} Q(x_{\sigma (m+1)}, \ldots, x_{\sigma (m+n)})],
\end{align}
for $P \in C^m_\mathrm{Hom} (\mathfrak{g}, \mathfrak{g})$, $Q \in C^n_\mathrm{Hom} (\mathfrak{g}, \mathfrak{g})$ and $x_1, \ldots, x_{m+n} \in \mathfrak{g}$. This is called the {\em cup product bracket} or simply the {\em cup product}.

\begin{remark}
When $\mathfrak{g}$ is a Lie algebra (i.e. $\alpha = \mathrm{id}_\mathfrak{g}$), one gets the cup product in the Chevalley-Eilenberg cochain complex of $\mathfrak{g}$ \cite{nij-ric-mor}.
\end{remark}

The cup product defined above can be generalized to cochains of a Hom-Lie algebra with coefficients in another Hom-Lie algebra. More precisely, let $(\mathfrak{g}, [~,~]_\mathfrak{g}, \alpha)$ and $(\mathfrak{h}, [~,~]_\mathfrak{h}, \beta)$ be two Hom-Lie algebras. For each ${m, n \geq 1}$, we define the cup product bracket (denoted by the same notation as above)
   $ [~,~]_{\mathsf{C}} : C^m_{\mathrm{Hom}} (\mathfrak{g}, \mathfrak{h}) \times C^n_{\mathrm{Hom}} (\mathfrak{g}, \mathfrak{h}) \rightarrow C^{m+n}_{\mathrm{Hom}} (\mathfrak{g}, \mathfrak{h})$
by
\begin{align}\label{cup-form}
    [P, Q]_{\mathsf{C}} (x_1, \ldots ,x_{m+n}) := \sum_{\sigma \in Sh (m,n)} (-1)^\sigma ~[\beta^{n-1} P (x_{\sigma (1)}, \ldots, x_{\sigma (m)}), \beta^{m-1} Q(x_{\sigma (m+1)}, \ldots, x_{\sigma (m+n)})]_\mathfrak{h},
\end{align}
for $P \in  C^m_{\mathrm{Hom}} (\mathfrak{g}, \mathfrak{h}),~ Q \in C^n_{\mathrm{Hom}} (\mathfrak{g}, \mathfrak{h})$ and  $x_1, \ldots, x_{m+n} \in \mathfrak{g}$. Then we have the following result.


\begin{thm}
Let $(\mathfrak{g},[~,~]_\mathfrak{g}, \alpha)$ and $(\mathfrak{h}, [~,~]_\mathfrak{h}, \beta) $ be two Hom-Lie algebras. Then 
\begin{align*}
( C^\bullet_\mathrm{Hom} (\mathfrak{g}, \mathfrak{h}) = \bigoplus_{{n \geq 1}} C^n_{\mathrm{Hom}} (\mathfrak{g}, \mathfrak{h}), [~,~]_{\mathsf{C}})
\end{align*}
is a graded Lie algebra.
\end{thm}

\begin{proof}
To check the graded skew-symmetry of the cup product bracket, we first observe a bijection between the sets $Sh (m,n)$ and $Sh (n,m)$.
Given any permutation $\sigma \in S_{m+n}$, we define a new permutation $\tau \in S_{m+n}$ by
\begin{align*}
    \tau (i) = \begin{cases}
        \sigma (m+i)  & \text{ for } 1 \leq i \leq n,\\
        \sigma (i-n)  & \text{ for } n+1 \leq i \leq m+n.
    \end{cases}
\end{align*}
Then it is easy to see that $\sigma \in Sh (m,n)$ if and only if $\tau \in Sh (n,m)$. Moreover, we have $(-1)^\sigma = (-1)^{mn} (-1)^\tau$.

For any $P \in  C^m_{\mathrm{Hom}} (\mathfrak{g}, \mathfrak{h})$ and $Q \in C^n_{\mathrm{Hom}} (\mathfrak{g}, \mathfrak{h})$, we observe that
\begin{align*}
     &[P, Q]_{\mathsf{C}} (x_1, \ldots ,x_{m+n})\\
     &= \sum_{\sigma \in Sh (m,n)} (-1)^\sigma ~[\beta^{n-1} P (x_{\sigma (1)}, \ldots, x_{\sigma (m)}), \beta^{m-1} Q(x_{\sigma (m+1)}, \ldots, x_{\sigma (m+n)})]_\mathfrak{h}\\
     &=(-1)^{mn} \sum_{\tau \in Sh (n,m)} (-1)^\tau ~[\beta^{n-1} P (x_{\tau (n+1)}, \ldots, x_{\tau (n+m)}), \beta^{m-1} Q(x_{\tau (1)}, \ldots, x_{\tau (n)})]_\mathfrak{h}\\
     &= -(-1)^{mn} \sum_{\tau \in Sh (n,m)} (-1)^\tau ~[\beta^{m-1} Q(x_{\tau (1)}, \ldots, x_{\tau (n)}), \beta^{n-1} P (x_{\tau (n+1)}, \ldots, x_{\tau (n+m)}) ]_\mathfrak{h}\\
     &=-(-1)^{mn}[ Q,P]_{\mathsf{C}} (x_1, \ldots, x_{m+n}).
\end{align*}
This shows that the cup product is graded skew-symmetric. 
To prove the graded Jacobi identity, we take $P \in C^m_{\mathrm{Hom}} (\mathfrak{g}, \mathfrak{h})$, $Q \in C^n_{\mathrm{Hom}} (\mathfrak{g}, \mathfrak{h})$ and $R \in C^k_{\mathrm{Hom}} (\mathfrak{g}, \mathfrak{h})$. Then for any $x_1, \ldots, x_{m+n+k} \in \mathfrak{g}$, we have
\begin{align} \label{PQR}
    &[P,[Q,R]_{\mathsf{C}}]_{\mathsf{C}}(x_1,x_2, \ldots , x_{m+n+k}) \nonumber\\
    &~= \sum_{\sigma \in Sh (m,n+k)} (-1)^\sigma~ [\beta^{n+k-1}P(x_{\sigma (1)},\ldots , x_{\sigma (m)}), \beta^{m-1}[Q,R]_{\mathsf{C}}(x_{\sigma (m+1)}, \ldots , x_{\sigma (m+n+k)})]_{\mathfrak{h}}  \nonumber  \\
    &~= \sum_{\sigma \in Sh (m,n+k)} (-1)^\sigma ~\big[\beta^{n+k-1}P(x_{\sigma (1)},\ldots , x_{\sigma (m)}),
    \beta^{m-1} \sum_{\tau \in Sh(n,k)}(-1)^\tau [\beta^{k-1}Q(x_{\tau \sigma (m+1)},\ldots ,x_{\tau \sigma (m+n)} ),  \nonumber \\
    & \qquad \qquad \qquad \qquad \qquad \qquad \qquad \qquad \qquad \qquad \qquad \qquad \qquad  \beta^{n-1}R(x_{\tau \sigma (m+n+1)},\ldots , x_{\tau \sigma (m+n+k)})]_{\mathfrak{h}} \big]_{\mathfrak{h}}  \nonumber \\
    &~=\sum_{\sigma \in Sh (m,n+k)} \sum_{\tau \in Sh(n,k)}(-1)^\sigma (-1)^\tau ~\big[\beta^{n+k-1}P(x_{\sigma (1)},\ldots , x_{\sigma (m)}), [\beta^{m+k-2}Q(x_{\tau \sigma (m+1)},\ldots ,x_{\tau \sigma (m+n)} ), \nonumber \\
     & \qquad \qquad \qquad \qquad \qquad \qquad \qquad \qquad \qquad \qquad  \beta^{m+n-2}R(x_{\tau \sigma (m+n+1)},\ldots , x_{\tau \sigma (m+n+k)})]_{\mathfrak{h}} \big]_{\mathfrak{h}}
    .
\end{align}
 We now establish a bijective correspondence between the sets $Sh(m,n,k)$ and $Sh(m,n+k) \times Sh(n,k)$. For any given permutations $\sigma \in S_{m+n+k}, \tau \in S_{n+k}$, another permutation, say $\gamma$ can be defined in $S_{m+n+k}$ as
\begin{align*}
    \gamma (i) = \begin{cases}
        \sigma (i)  & \text{ for } 1 \leq i \leq m,\\
        \tau \sigma (i)  & \text{ for } m+1 \leq i \leq m+n+k.
    \end{cases}
\end{align*}
If $\sigma \in Sh(m,n+k)$, $\tau \in Sh(n,k)$ then $\gamma \in Sh(m,n,k)$. We also see that $(-1)^{\gamma}= (-1)^{\sigma}(-1)^{\tau}$. Thus, from (\ref{PQR}) we get
    \begin{align*}
        &[P,[Q,R]_{\mathsf{C}}]_{\mathsf{C}}(x_1,x_2, \ldots , x_{m+n+k}) \\
        &~= \sum_{\gamma \in Sh (m,n,k)} (-1)^\gamma ~\big[\beta^{n+k-1}P(x_{\gamma (1)},\ldots , x_{\gamma (m)}),[\beta^{m+k-2}Q(x_{\gamma (m+1)},\ldots , x_{\gamma (m+n)}),\\
        & \qquad \qquad \qquad \qquad \qquad \qquad \qquad \qquad  \beta^{m+n-2}R(x_{\gamma (m+n+1)},\ldots , x_{\gamma (m+n+k)})]_{\mathfrak{h}} \big]_{\mathfrak{h}}.
    \end{align*}
By following the same approach and applying bijections between sets $Sh(m,n,k), Sh(n,k,m)$,  and then $Sh(m,n,k), Sh(k,m,n)$, we obtain the following two identities : 
    \begin{align*}
        &(-1)^{m(k+n)}[Q,[R,P]_{\mathsf{C}}]_{\mathsf{C}}(x_1,x_2, \ldots , x_{m+n+k})\\
        &~= \sum_{\sigma \in Sh (m,n,k)} (-1)^\sigma~\big [\beta^{m+k-1}Q(x_{\sigma (m+1)},\ldots , x_{\sigma (m+n)}),[\beta^{m+n-2}R(x_{\sigma (m+n+1)},\ldots , x_{\sigma (m+n+k)}),\\
        & \qquad \qquad \qquad \qquad \qquad \qquad  \beta^{k+n-2}P(x_{\sigma (1)},\ldots , x_{\sigma (m)})]_{\mathfrak{h}} \big]_{\mathfrak{h}},\\
        &(-1)^{k(m+n)}[R,[P,Q]_{\mathsf{C}}]_{\mathsf{C}}(x_1,x_2, \ldots , x_{m+n+k})\\
        &~= \sum_{\sigma \in Sh (m,n,k)} (-1)^\sigma~ \big[ \beta^{m+n-1}R(x_{\sigma (m+n+1)},\ldots , x_{\sigma (m+n+k)}),  [\beta^{n+k-2}P(x_{\sigma (1)},\ldots , x_{\sigma (m)}),\\
        &\qquad \qquad \qquad \qquad \qquad \qquad \beta^{m+k-2}Q(x_{\sigma (m+1)},\ldots , x_{\sigma (m+n)}) ]_{\mathfrak{h}} \big]_{\mathfrak{h}}.
    \end{align*}
Due to the Hom-Jacobi identity of $[~,~]_{\mathfrak{h}}$, it is clear that
    \begin{align*}
        [P,[Q,R]_{\mathsf{C}}]_{\mathsf{C}}+(-1)^{m(k+n)}[Q,[R,P]_{\mathsf{C}}]_{\mathsf{C}}+(-1)^{k(m+n)}[R,[P,Q]_{\mathsf{C}}]_{\mathsf{C}}=0.
    \end{align*}
Using the skew-symmetry of the $[~,~]_{\mathsf{C}}$ in the above expression, we finally get
    \begin{align*}
        [P,[Q,R]_{\mathsf{C}}]_{\mathsf{C}}= [[P,Q]_{\mathsf{C}},R]_{\mathsf{C}} + (-1)^{mn}[Q,[P,R]_{\mathsf{C}}]_{\mathsf{C}}.
    \end{align*}
Hence the bracket $[~,~]_\mathsf{C}$ satisfies the graded Jacobi identity.
\end{proof}


We also define a map $D : C^n_{\mathrm{Hom}} (\mathfrak{g}, \mathfrak{h}) \rightarrow C^{n+1}_{\mathrm{Hom}} (\mathfrak{g}, \mathfrak{h})$ by
\begin{align*}
(D f) (x_1, \ldots, x_{n+1}) = \sum_{1 \leq i < j \leq n+1} (-1)^{i+j}  ~ f ( [x_i, x_j]_\mathfrak{g}, \alpha (x_1), \ldots, \alpha (x_{i-1}), \widehat{\alpha (x_i)}, \ldots, \widehat{\alpha (x_j)}, \ldots, \alpha (x_{n+1} ) ),
\end{align*}
for $f \in  C^n_{\mathrm{Hom}} (\mathfrak{g}, \mathfrak{h})$ and $x_1, \ldots, x_{n+1} \in \mathfrak{g}$.
This is precisely the differential defining the cohomology of the Hom-Lie algebra $(\mathfrak{g}, [~,~]_\mathfrak{g}, \alpha)$ with coefficients in the trivial representation $(\mathfrak{h}, \diamond = 0, \beta)$. Hence we have $D^2 = 0$. Moreover, $D$ is a derivation for the cup product bracket $[~,~]_{\mathsf{C}}.$ In other words, $(C^\bullet_{\mathrm{Hom}} (\mathfrak{g}, \mathfrak{h}), [~,~]_\mathsf{C}, D) $ is a differential graded Lie algebra. In the following, we will study deformations of Hom-Lie algebra morphisms in terms of this differential graded Lie algebra.

\medskip

\noindent {\bf Cohomology and deformations of Hom-Lie algebra morphisms.}

    \begin{proposition}
        Let $(\mathfrak{g}, [~,~]_\mathfrak{g}, \alpha)$ and $(\mathfrak{h}, [~,~]_\mathfrak{h}, \beta)$ be two Hom-Lie algebras. Then a linear map $\varphi: \mathfrak{g} \rightarrow \mathfrak{h}$ is a Hom-Lie algebra morphism if and only if $\phi \in C^1_{\mathrm{Hom}} (\mathfrak{g}, \mathfrak{h})$ and it is a Maurer-Cartan element in the differential graded Lie algebra $(C^\bullet_{\mathrm{Hom}} (\mathfrak{g}, \mathfrak{h}), [~,~]_\mathsf{C}, D) $. 
    \end{proposition}

    \begin{proof}
    Note that the linear map $\varphi : \mathfrak{g} \rightarrow \mathfrak{h}$ satisfies $\beta \circ \varphi = \varphi \circ \alpha$ if and only if $\varphi \in C^1_\mathrm{Hom} (\mathfrak{g}, \mathfrak{h})$. Moreover, we have
        \begin{align*}
        \big(       D \varphi + \frac{1}{2} [\varphi, \varphi]_\mathsf{C} \big) (x, y ) = - \varphi ([x,y]_\mathfrak{g} ) + \frac{1}{2} ( [\varphi(x), \varphi(y)]_\mathfrak{h} - [\varphi(y), \varphi (x)]_\mathfrak{h}),
        \end{align*} 
        for any $x, y \in \mathfrak{g}$. Hence $\varphi$ is a Maurer-Cartan element (i.e. $ D \varphi + \frac{1}{2} [\varphi, \varphi]_\mathsf{C} = 0$) if and only if $\varphi ([x, y]_\mathfrak{g}) = [\phi(x) , \phi (y)]_\mathfrak{h}$, for $x, y \in \mathfrak{g}$.
    \end{proof}

    It follows from the above proposition that a morphism $\varphi : \mathfrak{g} \rightarrow \mathfrak{h}$ of Hom-Lie algebras induces a differential $D_\varphi : C^\bullet_{\mathrm{Hom}} (\mathfrak{g}, \mathfrak{h}) \rightarrow C^{\bullet + 1}_{\mathrm{Hom}} (\mathfrak{g}, \mathfrak{h})$ given by
    \begin{align}\label{twisted-diff}
    D_\varphi (f) = D f + [ \varphi, f ]_\mathsf{C}.
    \end{align}
    The triple $(C^\bullet_{\mathrm{Hom}} (\mathfrak{g}, \mathfrak{h}), [~,~]_\mathsf{C}, D_\varphi)$ is also a graded Lie algebra.

\begin{proposition}
Let $\varphi: \mathfrak{g} \rightarrow \mathfrak{h}$ be a morphism of Hom-Lie algebras. For a linear map $\varphi' : \mathfrak{g} \rightarrow \mathfrak{h}$ satisfying $\beta \circ \varphi' = \varphi' \circ \alpha$, the sum $\varphi + \varphi' $ is a Hom-Lie algebra morphism if and only if $\varphi' \in C^1_{\mathrm{Hom}} (\mathfrak{g}, \mathfrak{h})$ is a Maurer-Cartan element in the differential graded Lie algebra $( C^\bullet_{\mathrm{Hom}} (\mathfrak{g}, \mathfrak{h}), [~,~]_\mathsf{C}, D_\varphi )$.
\end{proposition}

\begin{proof}
We have
\begin{align*}
&D ( \varphi + \varphi' ) + \frac{1}{2} [\varphi + \varphi' ,  \varphi + \varphi']_{\mathsf{C}} \\
&= D \varphi + D \varphi' + \frac{1}{2} ( [ \varphi, \varphi]_{\mathsf{C}} +  [ \varphi, \varphi']_{\mathsf{C}} +  [ \varphi', \varphi]_{\mathsf{C}} + [ \varphi', \varphi']_{\mathsf{C}} ) \\
&= D \varphi' + [ \varphi, \varphi']_{\mathsf{C}} + \frac{1}{2} [\varphi', \varphi']_{\mathsf{C}} = D_\varphi (\varphi') +   \frac{1}{2} [\varphi', \varphi']_{\mathsf{C}}.
\end{align*}
Hence the result follows from definitions.
\end{proof}

Let $\varphi: \mathfrak{g} \rightarrow \mathfrak{h}$ be a morphism of Hom-Lie algebras. Then $\varphi$ induces a map $\diamond : \mathfrak{g} \times \mathfrak{h} \rightarrow \mathfrak{h}$ given by $x \diamond y := [ \varphi (x) , y ]_{\mathfrak{h}}$ for $x \in \mathfrak{g},~ y \in \mathfrak{h}.$ Then the triple $(\mathfrak{h}, \diamond , \beta)$ is a representation of the Hom-Lie algebra $(\mathfrak{g}, [~,~]_{\mathfrak{g}}, \alpha)$. Hence we can consider the cochain complex defining the cohomology of the Hom-Lie algebra $(\mathfrak{g}, [~,~]_\mathfrak{g}, \alpha)$ with coefficients in the above representation. Then it is easy to see that the corresponding differential map $\delta_\mathrm{Hom}^\varphi: C^\bullet_\mathrm{Hom} (\mathfrak{g}, \mathfrak{h}) \rightarrow  C^{\bullet +1}_\mathrm{Hom} (\mathfrak{g}, \mathfrak{h})$ is same with the map $D_\varphi$. The cohomology groups of the complex $\{ C^\bullet_\mathrm{Hom} (\mathfrak{g}, \mathfrak{h}), \delta_\mathrm{Hom}^\varphi \}$ or the complex $\{ C^\bullet_\mathrm{Hom} (\mathfrak{g}, \mathfrak{h}), D_\varphi   \}$ are called the {\em cohomology groups} of $\varphi$, and they are denoted by $H^\bullet_\varphi (\mathfrak{g}, \mathfrak{h})$. As an application of the cohomology of a Hom-Lie algebra morphism $\varphi$, we now study deformations of $\varphi$ by keeping the domain and codomain Hom-Lie algebras intact. The general deformations of a Hom-Lie algebra morphism are considered in \cite{arfa-fraj-makh}.

\medskip


Let $( \mathfrak{g}, [~,~]_\mathfrak{g}, \alpha)$ be a Hom-Lie algebra. Consider the space $\mathfrak{g} [ \! [ t ] \! ]$ of all formal power series in $t$ with coefficients in $\mathfrak{g}$. Then $\mathfrak{g}  [ \! [ t ] \! ]$ is a ${\bf k}  [ \! [ t ] \! ]$-module. Note that the ${\bf k}$-bilinear map $[~,~]_\mathfrak{g} : \mathfrak{g} \times \mathfrak{g} \rightarrow \mathfrak{g}$ (resp. the ${\bf k}$-linear map $\alpha : \mathfrak{g} \rightarrow \mathfrak{g}$) can be extended to a ${\bf k}  [ \! [ t ] \! ]$-bilinear map (resp. ${\bf k}  [ \! [ t ] \! ]$-linear map) which we denote by the same notation $[~,~]_\mathfrak{g} : \mathfrak{g}  [ \! [ t ] \! ]  \times \mathfrak{g} [ \! [ t ] \! ] \rightarrow \mathfrak{g}  [ \! [ t ] \! ]$ (resp. $\alpha : \mathfrak{g}  [ \! [ t ] \! ] \rightarrow \mathfrak{g}  [ \! [ t ] \! ]$). Then $(\mathfrak{g} [ \! [ t ] \! ], [~,~]_\mathfrak{g}, \alpha)$ is a Hom-Lie algebra over the ring ${\bf k} [ \! [ t ] \! ]$.

\begin{definition}
Let $( \mathfrak{g}, [~,~]_\mathfrak{g}, \alpha)$ and $( \mathfrak{h}, [~,~]_\mathfrak{h}, \beta)$ be two Hom-Lie algebras and $\varphi: \mathfrak{g} \rightarrow \mathfrak{h}$ be a Hom-Lie algebra morphism. A {\bf formal deformation} of $\varphi$ consists of a formal sum
\begin{align*}
\varphi_t = \sum_{i = 0}^\infty t^i \varphi_i \in \mathrm{Hom}(\mathfrak{g}, \mathfrak{h}) [\![t]\!] ~ \text{ with } \varphi_0 = \varphi
\end{align*}
that satisfies $\beta \circ \varphi_t = \varphi_t \circ \alpha$ and $\varphi_t ([x, y]_\mathfrak{g}) = [\varphi_t (x) , \varphi_t (y) ]_\mathfrak{h}$, for all $x, y \in \mathfrak{g}$.
\end{definition}

It follows that $\varphi_t = \sum_{i = 0}^\infty t^i \varphi_i $ is a formal deformation of $\varphi$ if the ${\bf k} [ \! [ t ] \! ]$-linear map $\varphi_t : \mathfrak{g}[\![ t ]\!] \rightarrow \mathfrak{h}[\![ t ]\!]$ is a morphism of Hom-Lie algebras from $( \mathfrak{g}[\![ t ]\!] , [~,~]_\mathfrak{g}, \alpha)$ to $( \mathfrak{h}[\![ t ]\!] , [~,~]_\mathfrak{h}, \alpha)$. Equivalently, we must have
\begin{align}\label{defor-eqn}
\begin{cases} \beta \circ \varphi_n = \varphi_n \circ \alpha,\\
\varphi_n ([x,y]_{\mathfrak{g}}) = \sum_{i+j= n} [\varphi_i (x) , \varphi_j (y) ]_\mathfrak{h},
\end{cases}
\end{align}
for each $n \geq 0$ and $x, y \in \mathfrak{g}$.
The conditions in (\ref{defor-eqn}) hold trivially for $n=0$ as we have $\varphi_0 = \varphi$ is a Hom-Lie algebra morphism. However, for $n=1$, we have $\beta \circ \varphi_1 = \varphi_1 \circ \alpha$ and
\begin{align}
\varphi_1 ([x, y]_\mathfrak{g}) = [\varphi(x) , \varphi_1 (y)]_\mathfrak{h} + [ \varphi_1 (x) , \varphi (y) ]_{\mathfrak{h}}, ~ \text{ for } x, y \in \mathfrak{g}.
\end{align}
This shows that $\varphi_1$ is a $1$-cocycle in the cohomology complex of $\varphi$. This is called the {\em infinitesimal} of the formal deformation $\varphi_t$.


In the following, we consider finite order deformations of a Hom-Lie algebra morphism and their extensibility. As before, let $(\mathfrak{g}, [~,~]_\mathfrak{g}, \alpha)$ and  $(\mathfrak{h}, [~,~]_\mathfrak{h}, \beta)$ be two Hom-Lie algebras and $\varphi : \mathfrak{g} \rightarrow \mathfrak{h}$ be a Hom-Lie algebra morphism. For $N \geq 1$, consider the Hom-Lie algebras $(\mathfrak{g} [\![t]\!] / (t^{N+1}) , [~,~]_\mathfrak{g}, \alpha )$ and $(\mathfrak{h} [\![t]\!] / (t^{N+1}) , [~,~]_\mathfrak{h}, \beta )$ over the ring ${\bf k} [\![t]\!] / (t^{N+1})$. An {\bf order $N$ deformation} of $\varphi$ consists of a sum $\varphi_t = \sum_{i = 0}^N t^i \varphi_i$ with $\varphi_0 = \varphi$ such that the ${\bf k} [\![t]\!] / (t^{N+1})$-linear map $\varphi_t$ is a Hom-Lie algebra morphism from $(\mathfrak{g} [\![t]\!] / (t^{N+1}) , [~,~]_\mathfrak{g}, \alpha ) $ to $(\mathfrak{h} [\![t]\!] / (t^{N+1}) , [~,~]_\mathfrak{h}, \beta )$.

Thus, $\varphi_t = \sum_{i=0}^N t^i \varphi_i$ is an order $N$ deformation if and only if the conditions in (\ref{defor-eqn}) hold for $n=0, 1, \ldots, N$. That is, 
\begin{align*}
    \beta \circ \varphi_n = \varphi_n \circ \alpha ~~~ \text{ and } ~~~ D(\varphi_n) + \frac{1}{2} \sum_{i+j = n} [\varphi_i , \varphi_j]_\mathsf{C} = 0,  ~ \text{ for } n = 0, 1, \ldots, N.
\end{align*}
The last condition can be equivalently written as
\begin{align}
D_\varphi (\varphi_n ) = - \frac{1}{2} \sum_{ \substack{   i+j = n \\ i, j \geq 1}} [\varphi_i, \varphi_j ]_\mathsf{C}, ~ \text{ for } n = 0, 1, \ldots, N.
\end{align}

An order $N$ deformation $\varphi_t = \sum_{i=0}^N t^i \varphi_i$ is said to be {\bf extensible} if there exists a linear map $\varphi_{N+1} : \mathfrak{g} \rightarrow \mathfrak{h}$ with $\beta \circ \varphi_{N+1} = \varphi_{N+1} \circ \alpha$ such that $\overline{\phi_t} = \sum_{i=0}^{N+1} t^i \varphi_i$ is an  order $N+1$ deformation. Thus, we are looking for a linear map $\varphi_{N+1} : \mathfrak{g} \rightarrow \mathfrak{h}$ with $\beta \circ \varphi_{N+1} = \varphi_{N+1} \circ \alpha$ that satisfies
\begin{align}\label{exten-eqn}
D_\varphi ( \varphi_{N+1} ) = - \frac{1}{2} \sum_{ \substack{   i+j = N+1 \\ i, j \geq 1}} [ \varphi_i, \varphi_j ]_\mathsf{C}.
\end{align}
Note that the right-hand side of (\ref{exten-eqn}) depends only on the deformation $\varphi_t$. It is called the {\em obstruction} to extend the order $N$ deformation $\varphi_t$, and denoted by $\mathrm{Ob}_{\varphi_t}.$

\begin{proposition}
The obstruction $\mathrm{Ob}_{\varphi_t}$ is a $2$-cocycle in the cohomology complex of $\varphi$, i.e. 
\begin{align*}
D_\varphi (  \mathrm{Ob}_{\varphi_t}) = 0.
\end{align*}
\end{proposition}

\begin{proof}
We have
\begin{align*}
D_\varphi (  \mathrm{Ob}_{\varphi_t})  &= D_\varphi ( - \frac{1}{2} \sum_{ \substack{   i+j = N+1 \\ i, j \geq 1}} [ \varphi_i, \varphi_j ]_\mathsf{C} ) \\
&=  - \frac{1}{2} \sum_{ \substack{   i+j = N+1 \\ i, j \geq 1}}   ( D [ \varphi_i, \varphi_j ]_\mathsf{C} + [ \varphi, [ \varphi_i, \varphi_j ]_\mathsf{C}  ]_\mathsf{C}) \\
&= - \frac{1}{2} \sum_{ \substack{   i+j = N+1 \\ i, j \geq 1}}   \big( [ D \varphi_i, \varphi_j ]_\mathsf{C} -  [ \varphi_i, D \varphi_j ]_\mathsf{C}  + [ [ \varphi, \varphi_i]_\mathsf{C} , \varphi_j ]_\mathsf{C}  - [ \varphi_i, [ \varphi, \varphi_j ]_\mathsf{C} ]_\mathsf{C} \big) \\
&=  - \frac{1}{2} \sum_{ \substack{   i+j = N+1 \\ i, j \geq 1}}    \big( [ D_\varphi (\varphi_i), \varphi_j ]_\mathsf{C} -  [ \varphi_i, D_\varphi (\varphi_j) ]_\mathsf{C} \big)  \\
&= \frac{1}{4} \sum_{ \substack{   i_1 + i_2 +j = N+1 \\ i_1, i_2, j \geq 1}} [[ \varphi_{i_1}, \varphi_{i_2}]_\mathsf{C}, \varphi_j ]_\mathsf{C}  - \frac{1}{4} \sum_{ \substack{   i+j_1 + j_2 = N+1 \\ i, j_1, j_2 \geq 1}} [ \varphi_i, [ \varphi_{j_1}, \varphi_{j_2}]_\mathsf{C} ]_\mathsf{C} \\
&= \frac{1}{2} \sum_{ \substack{   i+j+k = N+1 \\ i, j, k \geq 1}} [[ \varphi_{i}, \varphi_{j} ]_\mathsf{C}, \varphi_k ]_\mathsf{C} = 0.
\end{align*}
Hence the proof.
\end{proof}

The cohomology class $[ \mathrm{Ob}_{\varphi_t} ] \in {H^2_\varphi ( \mathfrak{g}, \mathfrak{h})} $ is called the {\em obstruction class} to extend the deformation $\varphi_t$.  As a consequence of (\ref{exten-eqn}), we now get the following result.

\begin{thm}
A finite order deformation $\varphi_t$ of a Hom-Lie algebra morphism $\varphi$ is extensible if and only if the corresponding obstruction class $[ \mathrm{Ob}_{\varphi_t} ] \in {H^2_\varphi ( \mathfrak{g}, \mathfrak{h})}$ vanishes.
\end{thm}

\begin{corollary}
(i) If ${H^2_\varphi ( \mathfrak{g}, \mathfrak{h})} = 0$ then any finite order deformation of $\varphi$ is extensible.

(ii) If ${H^2_\varphi ( \mathfrak{g}, \mathfrak{h})} = 0$ then every $1$-cocycle in the cohomology complex of $\varphi$ is the infinitesimal of some formal deformation of $\varphi$.
\end{corollary}

\medskip

\noindent {\bf Some other descriptions of the cup product bracket.}
Let $(\mathfrak{g}, [~,~], \alpha)$ be a Hom-Lie algebra. Here we give two new descriptions of the cup product operation (\ref{cupp}). Among others, we show that the cup product operation on cochains induces a trivial operation at the level of cohomology.


First, we define a map $\theta : C^n_\mathrm{Hom}(\mathfrak{g}, \mathfrak{g}) \rightarrow  C^{n+1}_\mathrm{Hom}(\mathfrak{g}, \mathfrak{g})$ by
\begin{align} \label{theta}
\theta f :=  - i_{f}  \mu, \text{ for } f \in C^n_{\mathrm{Hom}} (\mathfrak{g}, \mathfrak{g}).
\end{align}
Here $\mu \in C^2_{\mathrm{Hom}} (\mathfrak{g}, \mathfrak{g})$ is the element that corresponds to the Hom-Lie bracket of $\mathfrak{g}$.  Explicitly, we have
\begin{align*}
(\theta f) (x_1, \ldots, x_{n+1})  =  \sum_{i=1}^{n+1} (-1)^{n+i} ~[ f (x_1, \ldots, \widehat{ x_i}, \ldots, x_{n+1} ) ,\alpha^{n-1} (x_i) ],
\end{align*}
for $f \in C^n_\mathrm{Hom} (\mathfrak{g}, \mathfrak{g})$ and $x_1, \ldots , x_{n+1}\in \mathfrak{g}$. With this notation, we have the following.


\begin{lemma}\label{lemma-lemma}
Let $(\mathfrak{g}, [~,~], \alpha)$ be a Hom-Lie algebra. For any $P \in C^m_{\mathrm{Hom}} (\mathfrak{g}, \mathfrak{g})$ and $Q \in C^n_{\mathrm{Hom}} (\mathfrak{g}, \mathfrak{g})$, we have
\begin{align*}
[P, Q]_{\mathsf{C}} = (-1)^n \big(  i_P  (\theta Q) - \theta (i_P  Q) \big).
\end{align*}
\end{lemma}

\begin{proof}
For any $x_1, \ldots , x_{m+n}\in \mathfrak{g}$, we have
\begin{align} 
    & \theta (i_P Q)(x_1, \ldots , x_{m+n})  \nonumber\\ \nonumber
    ~&~ = -\sum_{\sigma \in Sh(m+n-1,1)} (-1)^\sigma ~[(i_P Q)(x_{\sigma (1)},\ldots , x_{\sigma (m+n-1)}),\alpha^{m+n-2} x_{\sigma (m+n)}]\\ \nonumber
    ~&~ =- \sum_{\sigma \in Sh(m+n-1,1)} (-1)^\sigma \sum_{\tau \in Sh(m,n-1)} (-1)^\tau ~ \big[Q\big(P(x_{\tau \sigma (1)}, \ldots ,x_{\tau \sigma (m)} ),\alpha^{m-1}x_{\tau \sigma (m+1)}, \ldots   \big),\\ \nonumber
    ~&~\qquad \qquad \qquad \qquad \qquad \qquad \qquad \qquad \qquad \qquad   \alpha^{m+n-2}x_{\tau \sigma (m+n)}\big]\\ 
    ~&~ = - \sum_{\gamma \in Sh(m,n-1,1)} (-1)^\gamma ~ \big[ Q\big(P(x_{\gamma (1)}, \ldots , x_{\gamma (m)}),\alpha^{m-1}x_{ \gamma (m+1)}, \ldots , \alpha^{m-1} x_{ \gamma (m+n-1)}  \big), \alpha^{m+n-2}x_{ \gamma (m+n)}    \big]. \label{iPQ}
\end{align}
On the other hand,
\begin{align}
   & (i_P(\theta Q)) (x_1, \ldots , x_{m+n}) \nonumber \\
    ~& = -i_P (i_Q \mu) (x_1, \ldots , x_{m+n})  \nonumber\\
    ~& = -\sum_{\sigma \in Sh(m,n)} (-1)^\sigma ~(i_Q \mu) \big( P(x_{\sigma (1)}, \ldots , x_{\sigma (m)}), \alpha^{m-1} x_{\sigma (m+1)}, \ldots ,\alpha^{m-1} x_{\sigma (m+n)}   \big) \nonumber\\
    ~& =-  \sum_{\sigma \in Sh(m,n)} (-1)^\sigma ~\Big(  \sum_{\tau \in Sh(n-1,1)} (-1)^\tau ~\big[ Q \big(  P(x_{\sigma (1)}, \ldots , x_{\sigma (m)}), \ldots , \alpha^{m-1} x_{\tau \sigma (m+n-1)}  \big), 
      \alpha^{m+n-2}x_{\tau \sigma (m+n)}\big]  \nonumber  \\
      & \qquad \qquad \qquad \qquad \qquad \qquad  + (-1)^n \big[ \alpha^{m-1}Q(x_{\sigma (m+1)}, \ldots , x_{\sigma (m+n)}), \alpha^{n-1}P(x_{\sigma (1)}, \ldots , x_{\sigma (m)})  \big]    \Big) \nonumber \\
    ~&= -  \sum_{\sigma \in Sh(m,n)} (-1)^\sigma \sum_{\tau \in Sh(n-1,1)} (-1)^\tau~ \big[ Q \big(  P(x_{\sigma (1)}, \ldots , x_{\sigma (m)}),  \alpha^{m-1} x_{\tau \sigma (m+1)}, \ldots  \big), \alpha^{m+n-2}x_{\tau \sigma (m+n)}\big] \nonumber \\
    & \qquad  \qquad \qquad  - (-1)^n \sum_{\sigma \in Sh(m,n)} (-1)^\sigma~ \big[ \alpha^{m-1}Q(x_{\sigma (m+1)}, \ldots , x_{\sigma (m+n)}), \alpha^{n-1}P(x_{\sigma (1)}, \ldots , x_{\sigma (m)})  \big] \nonumber \\
    ~& = -\sum_{\gamma \in Sh(m,n-1,1)} (-1)^\gamma ~ \big[ Q\big(P(x_{\gamma (1)}, \ldots , x_{\gamma (m)}),\alpha^{m-1}x_{ \gamma (m+1)}, \ldots , \alpha^{m-1} x_{ \gamma (m+n-1)}  \big), \alpha^{m+n-2}x_{ \gamma (m+n)}    \big]  \nonumber \\
    ~& \qquad \qquad \qquad +(-1)^n \sum_{\sigma \in Sh(m,n)} (-1)^\sigma ~\big[ \alpha^{n-1}P(x_{\sigma (1)}, \ldots , x_{\sigma (m)}) ,\alpha^{m-1}Q(x_{\sigma (m+1)}, \ldots , x_{\sigma (m+n)}) \big] \nonumber \\
    ~& = ~\theta (i_P Q)(x_1, \ldots , x_{m+n}) +(-1)^n [P,Q]_{\mathsf{C}}(x_1, \ldots , x_{m+n})   \qquad  (\text{using (\ref{iPQ}) and the definition of $[~,~]_{\mathsf{C}}$}). \nonumber
\end{align}
Hence we can conclude that $[P,Q]_{\mathsf{C}}= (-1)^n \big(i_P \theta Q- \theta (i_P Q)  \big)$.
\end{proof}
 
 \begin{proposition}\label{cup-new-new}
 For any $P \in C^m_{\mathrm{Hom}} (\mathfrak{g}, \mathfrak{g})$ and $Q \in C^n_{\mathrm{Hom}} (\mathfrak{g}, \mathfrak{g})$, we have
 \begin{align*}
 [P, Q ]_{\mathsf{C}} = i_P (\delta_{\mathrm{Hom}} Q ) + (-1)^{m-1}i_{ (\delta_{\mathrm{Hom}} P)} Q  + (-1)^m ~\delta_{\mathrm{Hom}} ( i_P Q).
 \end{align*}
 \end{proposition}
 
 \begin{proof}

Expanding the RHS, we get
    \begin{align*}
        &i_P (\delta_{\mathrm{Hom}} Q ) + (-1)^{m-1}i_{ (\delta_{\mathrm{Hom}} P)} Q  + (-1)^m ~\delta_{\mathrm{Hom}} ( i_P Q)\\
        &=-i_P [\mu , Q]_{\mathsf{NR}} - (-1)^{m-1} i_{[\mu , P]_{\mathsf{NR}}}Q - (-1)^m [\mu, i_PQ]_{\mathsf{NR}}\\
        &= -i_P i_{\mu} Q + (-1)^{n-1} i_P i_Q \mu - (-1)^{m-1} i_{i_{\mu}P}Q + i_{i_P \mu}Q - (-1)^m i_{\mu} i_P Q +(-1)^m(-1)^{m+n-2}i_{i_P Q}\mu\\
        &= (-1)^n \big( -i_P i_Q \mu + i_{i_P Q}\mu \big) +i_{i_P \mu} Q- i_P i_{\mu} Q -(-1)^{m-1}\big( i_{i_{\mu}P}Q - i_{\mu} i_P Q \big)\\
        &= (-1)^n \big(i_P \theta Q - \theta (i_P Q)\big) \qquad (\text{by the graded right pre-Lie identity of $P\odot Q= i_Q P$})\\
        &= [P,Q]_\mathsf{C} \quad (\text{by Lemma \ref{lemma-lemma}}).
    \end{align*}
This proves the result.
 \end{proof}

 
 \begin{proposition}\label{delta-hom-derivation}
 For any $P \in C^m_{\mathrm{Hom}} (\mathfrak{g}, \mathfrak{g})$ and $Q \in C^n_{\mathrm{Hom}} (\mathfrak{g}, \mathfrak{g})$, we have
 \begin{align*}
 \delta_{\mathrm{Hom}} [P, Q]_{\mathsf{C}} =  [\delta_{\mathrm{Hom}}  P, Q]_{\mathsf{C}} + (-1)^m [P, \delta_{\mathrm{Hom}} Q]_{\mathsf{C}}.
 \end{align*}
 \end{proposition}
 
 \begin{proof}
From Proposition \ref{cup-new-new}, we have
\begin{align*}
\delta_{\mathrm{Hom}} [P,Q]_\mathsf{C} = \delta_{\mathrm{Hom}}  (i_P (\delta_{\mathrm{Hom}} Q )) + (-1)^{m-1} \delta_{\mathrm{Hom}} ( i_{(\delta_{\mathrm{Hom}} P)} Q ).
\end{align*}
On the other hand, by applying the same proposition, we get
\begin{align*}
&[\delta_{\mathrm{Hom}} P, Q ]_\mathsf{C} ~+~ (-1)^m ~[ P, \delta_{\mathrm{Hom}} Q ]_\mathsf{C} \\
&= \cancel{i_{(\delta_{\mathrm{Hom}} P)}(\delta_{\mathrm{Hom}} Q)} + (-1)^{m+1} \delta_{\mathrm{Hom}}  ( i_{(\delta_{\mathrm{Hom}} P)}Q)  \\
& \qquad \qquad + (-1)^m (-1)^{m-1} \cancel{ i_{(\delta_{\mathrm{Hom}} P)}(\delta_{\mathrm{Hom}} Q)}  + (-1)^m (-1)^m \delta_{\mathrm{Hom}} (  i_P (\delta_\mathrm{Hom} Q)) \\
&= \delta_{\mathrm{Hom}} (  i_P (\delta_\mathrm{Hom} Q)) + (-1)^{m+1} \delta_{\mathrm{Hom}}  ( i_{(\delta_{\mathrm{Hom}} P)} Q  ).
\end{align*}
Hence, the result follows.
 \end{proof}

 The above proposition says that the differential $\delta_\mathrm{Hom}$ is a graded derivation for the cup product bracket $[~,~]_\mathsf{C}$. In other words, $( C^\bullet_\mathrm{Hom} (\mathfrak{g}, \mathfrak{g}), [~,~]_\mathsf{C}, \delta_\mathrm{Hom})$ is a differential graded Lie algebra. In particular, it says that the cup product bracket $[~,~]_\mathsf{C}$ induces an operation on the graded space of cohomology $H^\bullet_\mathrm{Hom} (\mathfrak{g}, \mathfrak{g})$. However, as a consequence of Proposition \ref{cup-new-new}, we obtain the following result.

 \begin{thm}
     Let $(\mathfrak{g}, [~,~], \alpha)$ be a Hom-Lie algebra. Then the induced cup product operation at the level of cohomology is trivial.
 \end{thm}

 When $\mathfrak{g}$ is a Lie algebra (i.e. $\alpha = \mathrm{id}_\mathfrak{g}$), we recover the fact that the induced cup product operation at the level of Chevalley-Eilenberg cohomology is trivial \cite{nij-ric-mor}.
 

\medskip

Let $\delta^\mathrm{tr}_\mathrm{Hom} : C^n_\mathrm{Hom} (\mathfrak{g}, \mathfrak{g}) \rightarrow C^{n+1}_\mathrm{Hom} (\mathfrak{g}, \mathfrak{g})$, for $n \geq 0$, be the differential of the Hom-Lie algebra $(\mathfrak{g}, [~,~], \alpha)$ with coefficients in $\mathfrak{g}$ but with the trivial action. Then we have $\delta^\mathrm{tr}_\mathrm{Hom} (f) = - i_\mu f$, for $f \in C^n_\mathrm{Hom} (\mathfrak{g}, \mathfrak{g})$. It is easy to see that $\delta^\mathrm{tr}_\mathrm{Hom} $ is also a graded derivation for the cup product bracket. Note that, we have $\delta_\mathrm{Hom} f = \delta_\mathrm{Hom}^\mathrm{tr} f + (-1)^n \theta f$, for $f \in C^n_\mathrm{Hom} (\mathfrak{g}, \mathfrak{g})$. That is,
\begin{align*}
    \theta f = (-1)^n (\delta_\mathrm{Hom} - \delta_\mathrm{Hom}^\mathrm{tr}) (f), \text{ for } f \in C^n_\mathrm{Hom} (\mathfrak{g}, \mathfrak{g}).
\end{align*}
Since $\delta_\mathrm{Hom}$ and $\delta_\mathrm{Hom}^\mathrm{tr}$ are both graded derivations for the cup product bracket, we have the following derivation-like property of the map $\theta$.
This result will be useful in constructing the derived bracket in Section \ref{sec5}.

\begin{proposition}\label{tr-theta}
For any $P \in C^m_\mathrm{Hom} (\mathfrak{g}, \mathfrak{g}),~ Q \in C^n_\mathrm{Hom}(\mathfrak{g}, \mathfrak{g}),$ we have 
\begin{align*}
\theta  [P, Q]_{\mathsf{C}}  =~& (-1)^n [ \theta (P), Q]_{\mathsf{C}} + [P, \theta (Q)]_{\mathsf{C}}.
\end{align*}
\end{proposition} 





\section{Fr\"olicher-Nijenhuis bracket for Hom-Lie algebras} \label{sec4}

Given a Hom-Lie algebra, we first show a graded Lie algebra action of the Nijenhuis-Richardson graded Lie algebra on the cup product graded Lie algebra. As a byproduct, we obtain a new graded Lie bracket (which we call the Fr\"{o}licher-Nijenhuis bracket) on the space of cochains of a Hom-Lie algebra. Next, we study Nijenhuis operators on a Hom-Lie algebra in terms of the Fr\"{o}licher-Nijenhuis bracket. Finally, we show that the Nijenhuis-Richardson algebra and the Fr\"{o}licher-Nijenhuis algebra form a matched pair of graded Lie algebras.

Let $(\mathcal{A}, [~,~]_\mathcal{A})$ and $(\mathcal{B}, [~,~]_\mathcal{B})$ be two graded Lie algebras. Then a {\em graded Lie algebra action} of $(\mathcal{A}, [~,~]_\mathcal{A})$ on $(\mathcal{B}, [~,~]_\mathcal{B})$ is a degree $0$ bilinear map $\rho: \mathcal{A} \times \mathcal{B} \rightarrow \mathcal{B},~ (a,b) \mapsto \rho (a) (b)$ that satisfies
\begin{align}
\rho ([a, a']_\mathcal{A})(b) =~& \rho (a) \rho (a')(b) - (-1)^{|a||a'|}  \rho (a') \rho (a)(b), \label{action1}\\
\rho (a) ([b, b']_\mathcal{B}) =~& [ \rho (a) b, b']_\mathcal{B} + (-1)^{|a||b|} [b, \rho (a) b']_\mathcal{B}, \label{action2}
\end{align}
for all homogeneous elements $a, a' \in \mathcal{A}$ and $b, b' \in \mathcal{B}$. The first condition says that the map $\mathcal{A} \rightarrow \mathrm{End} (\mathcal{B})$ is a morphism of graded Lie algebras, while the second condition implies that $\rho(a)$ is a graded derivation of degree $|a|$ on the graded Lie algebra $(\mathcal{B}, [~,~]_\mathcal{B})$. Thus, a graded Lie algebra action as above is simply given by a graded Lie algebra morphism $\mathcal{A} \rightarrow \mathrm{Der} ( \mathcal{B}, [~,~]_\mathcal{B})$.

The above notion of a graded Lie algebra action is a generalization of the Lie algebra action of a Lie algebra on another Lie algebra. It is well-known that a Lie algebra action gives rise to the semidirect product. The corresponding graded version is given by the following.

\begin{proposition}\label{semidirect} Let $(\mathcal{A}, [~,~]_\mathcal{A})$ and $(\mathcal{B}, [~,~]_\mathcal{B})$ be two graded Lie algebras and let $\rho : \mathcal{A} \times \mathcal{B} \rightarrow \mathcal{B}$ defines an action of the graded Lie algebra $\mathcal{A}$ on $\mathcal{B}$. Then the direct sum $\mathcal{A} \oplus \mathcal{B}$ carries a graded Lie algebra structure with the bracket given by
\begin{align*}
[(a, b), (a', b')]_\ltimes := ( [a, a']_\mathcal{A}, ~[b, b']_\mathcal{B} + \rho (a) b' - (-1)^{|b||a'|} \rho (a') b ),
\end{align*} 
for homogeneous elements $(a, b), (a', b') \in \mathcal{A} \oplus \mathcal{B}$.
\end{proposition}


The graded Lie algebra constructed in the above proposition is called the {\em semidirect product}, and it is denoted by $(\mathcal{A}, [~,~]_\mathcal{A}) \ltimes (\mathcal{B}, [~,~]_\mathcal{B})$ or simply by $\mathcal{A} \ltimes \mathcal{B}$ when the brackets are clear from the context.

\medskip

Let $(\mathfrak{g}, [~,~]_\mathfrak{g}, \alpha)$ and $(\mathfrak{h}, [~,~]_\mathfrak{h}, \beta)$ be two Hom-Lie algebras. Consider the graded Lie algebras $(C^{\bullet +1}_{\mathrm{Hom}} (\mathfrak{g}, \mathfrak{g}), [~,~]_{\mathsf{NR}})$ and $(C^{\bullet}_{\mathrm{Hom}} (\mathfrak{g}, \mathfrak{h}), [~,~]_{\mathsf{C}})$ obtained in Sections \ref{sec2} and \ref{sec3}, respectively. We define a degree $0$ bilinear map
\begin{align}\label{first-semi}
\rho : C^{\bullet +1}_{\mathrm{Hom}} (\mathfrak{g}, \mathfrak{g}) \times C^{\bullet}_{\mathrm{Hom}} (\mathfrak{g}, \mathfrak{h}) \rightarrow C^{\bullet}_{\mathrm{Hom}} (\mathfrak{g}, \mathfrak{h}), ~(P, E) \mapsto \rho(P) (E) := i_P E,
\end{align}
for $P \in C^{m +1}_{\mathrm{Hom}} (\mathfrak{g}, \mathfrak{g})$ and $ E \in C^{k}_{\mathrm{Hom}} (\mathfrak{g}, \mathfrak{h})$. Then we have the following.

\begin{proposition}\label{first-semi-rep}
The map (\ref{first-semi}) defines an action of the graded Lie algebra $(C^{\bullet +1}_{\mathrm{Hom}} (\mathfrak{g}, \mathfrak{g}), [~,~]_{\mathsf{NR}})$ on the graded Lie algebra $(C^{\bullet}_{\mathrm{Hom}} (\mathfrak{g}, \mathfrak{h}), [~,~]_{\mathsf{C}})$.
\end{proposition}

\begin{proof}
Let $P \in C^{m+1}_\mathrm{Hom}(\mathfrak{g}, \mathfrak{g})$, $Q \in C^{n+1}_\mathrm{Hom} (\mathfrak{g}, \mathfrak{g})$, $E \in C^k_\mathrm{Hom} (\mathfrak{g}, \mathfrak{h})$ and $F \in C^l_\mathrm{Hom} (\mathfrak{g}, \mathfrak{h})$. Then we have
\begin{align*}
    \rho ( [P, Q]_\mathsf{NR}) (E) &= i_{( [P, Q]_\mathsf{NR})} E\\
    &= i_{i_P Q} E - (-1)^{mn} ~i_{i_Q P}E\\
    &= i_P i_Q E- (-1)^{mn} i_Q i_P E \qquad \text{(by the graded right pre-Lie identity of $P \odot Q= i_Q P$)}\\
    &= \rho(P) \rho(Q) (E) - (-1)^{mn} \rho(Q) \rho (P) (E) .
\end{align*}
Hence the condition (\ref{action1}) holds.
Next, for any $x_1 ,\ldots , x_{k+l+m} \in \mathfrak{g}$, we first observe that
\begin{align}
    &[ \rho (P) E, F]_\mathsf{C} (x_1,\ldots , x_{k+l+m}) \nonumber\\
    &~= [i_P E,F]_\mathsf{C} (x_1,\ldots , x_{k+l+m}) \nonumber \\
    &~= \sum_{\sigma \in Sh (m+k, l)} (-1)^\sigma ~\big[ \beta^{l-1}(i_P E)(x_{\sigma (1)}, \ldots, x_{\sigma (m+k)}), \beta^{m+k-1} F(x_{\sigma (m+k+1)}, \ldots, x_{\sigma (m+k+l)})   \big]_{\mathfrak{h}} \nonumber \\
    &~= \sum_{\sigma \in Sh (m+k, l)} (-1)^\sigma \sum_{\tau \in Sh (m+1, k-1)} (-1)^\tau~ \big[ \beta^{l-1}E\big( P(x_{\tau \sigma (1)}, \ldots, x_{\tau \sigma (m+1)}),\alpha^m x_{\tau \sigma (m+2)},\ldots , \alpha^m x_{\tau \sigma (m+k)}  \big), \nonumber \\
    &~ \qquad \qquad \qquad \qquad \qquad \qquad \qquad \qquad   \beta^{m+k-1}F(x_{\sigma (m+k+1)}, \ldots, x_{\sigma (m+k+l)}  )  \big]_{\mathfrak{h}} \nonumber \\
    &~= \sum_{\gamma \in Sh (m+1,k-1, l)} (-1)^\gamma  ~\big[ \beta^{l-1}E\big( P(x_{ \gamma  (1)}, \ldots, x_{ \gamma  (m+1)}),\alpha^m x_{ \gamma  (m+2)},\ldots , \alpha^m x_{\gamma  (m+k)}  \big), \nonumber\\
    &~ \qquad \qquad \qquad \qquad \qquad \qquad \qquad \qquad   \beta^{m+k-1}F(x_{\gamma  (m+k+1)}, \ldots, x_{\gamma  (m+k+l)}  )  \big]_{\mathfrak{h}}. \label{action-1}
\end{align}
Using similar calculations,  we get
\begin{align*}
    &[ E,\rho (P)  F]_\mathsf{C}(x_1,\ldots , x_{k+l+m})\\
    &~= [E, i_P F]_\mathsf{C} (x_1,\ldots , x_{k+l+m})\\
    &~= \sum_{\sigma \in Sh (k, l+m)} (-1)^\sigma ~\big[ \beta^{l+m-1}E(x_{\sigma (1)}, \ldots , x_{\sigma (k)}), \beta^{k-1}(i_P F)(x_{\sigma (k+1)}, \ldots , x_{\sigma (k+l+m)})  \big]_{\mathfrak{h}}\\
    &~= \sum_{\sigma \in Sh (k, l+m)} (-1)^\sigma \sum_{\tau \in Sh (m+1, l-1)} (-1)^\tau~ \big[ \beta^{l+m-1}E(x_{\sigma (1)}, \ldots , x_{\sigma (k)}), \beta^{k-1} F \big(P(x_{\tau \sigma (k+1)},\ldots , x_{\tau \sigma (k+m+1)}),\\
    &~  \qquad \qquad  \qquad \qquad  \qquad \qquad \qquad \qquad  \alpha^m x_{\tau \sigma (k+m+2)}, \ldots ,\alpha^m x_{\tau \sigma (k+m+l)}  \big) \big]_{\mathfrak{h}}\\
    &~= \sum_{\gamma  \in Sh (k, m+1, l-1)} (-1)^\gamma ~ \big[ \beta^{l+m-1}E(x_{\gamma  (1)}, \ldots , x_{\gamma  (k)}), \beta^{k-1} F \big(P(x_{\gamma  (k+1)},\ldots , x_{ \gamma  (k+m+1)}),\\
    &~  \qquad \qquad   \qquad \qquad  \qquad \qquad \qquad \qquad  \alpha^m x_{ \gamma  (k+m+2)}, \ldots ,\alpha^m x_{\gamma (k+m+l)}  \big) \big]_{\mathfrak{h}}.
\end{align*}
    Changing the order of the shuffle in the above equation, we get 
        \begin{align}
            &(-1)^{(m+1)k}[ E,\rho (P)  F]_\mathsf{C}(x_1,\ldots , x_{k+l+m}) \nonumber\\
            &~= \sum_{\gamma \in Sh (m+1,k, l-1)} (-1)^\gamma~ \big[ \beta^{l+m-1}E(x_{\gamma (m+2)}, \ldots , x_{\gamma (m+k+1)}), \beta^{k-1} F \big(P(x_{\gamma (1)},\ldots , x_{ \gamma (m+1)}), \nonumber\\
            &~  \qquad \qquad   \qquad \qquad \qquad \qquad \qquad \qquad  \alpha^m x_{ \gamma (k+m+2)}, \ldots ,\alpha^m x_{\gamma (k+m+l)}  \big) \big]_{\mathfrak{h}}. \label{action-2}
        \end{align}
    Before proceeding further, note that for any $m,n \geq 1$, $$Sh(m,n)= \{\sigma \in Sh(m,n) |~ \! \sigma(1)=1 \} \bigsqcup~ \{\sigma \in Sh(m,n) |~ \! \sigma(m+1)=1 \} .$$
    Given any permutation $\sigma \in S_{m+n}$ with $\sigma(1)=1$, we can define a new permutation $\tau \in S_{m+n-1}$ by
        \begin{align*}
            \tau(i) = \sigma(i+1)-1 \quad \text{for   } 1\leq i \leq m+n-1.
        \end{align*}
    It is easy to see that $\sigma \in Sh(m,n)$ if and only if $\tau \in Sh(m-1, n)$. Also note that $(-1)^{\sigma}= (-1)^\tau$. On the other hand, if $\sigma \in S_{m+n}$ with $\sigma(m+1)=1$, we can define a permutation $\gamma \in S_{m+n-1}$ by 
        \begin{align*}
            \gamma (i) = \begin{cases}
        \sigma (i)-1  & \text{ for } 1 \leq i \leq m,\\
        \sigma (i+1)-1  & \text{ for } m+1 \leq i \leq m+n-1.
        \end{cases}
        \end{align*}
    Then we observe that $\sigma \in Sh(m,n)$ if and only if $\gamma \in Sh(m,n-1)$ and moreover $(-1)^\sigma =(-1)^m (-1)^\gamma $. Thus, the set $Sh(m,n)$ can be written as a disjoint union of $Sh(m-1,n)$ and $Sh(m,n-1)$. Using this observation, we have
\begin{align*}
    &\rho (P) [E, F]_\mathsf{C} (x_1,\ldots , x_{k+l+m})\\
    &~= (i_P [E, F]_\mathsf{C}) (x_1,\ldots , x_{k+l+m})\\
    &~= \sum_{\sigma \in Sh (m+1,k+ l-1)} (-1)^\sigma~ [E,F]_\mathsf{C}\big(P( x_{\sigma (1)},\ldots , x_{ \sigma (m+1)} ),\alpha^m x_{\sigma (m+2)}, \ldots ,\alpha^m x_{\sigma (m+k+l)}   \big)\\
    &~= \sum_{\sigma \in Sh (m+1,k+ l-1)} (-1)^\sigma~ \Big( \sum_{\tau \in Sh (k-1, l)} (-1)^\tau~ \big[ \beta^{l-1} E\big(P(x_{\sigma (1)},\ldots , x_{ \sigma (m+1)}),\alpha^m x_{\tau \sigma (m+2)}, \ldots , \alpha^m x_{\tau \sigma (m+k)}\big),\\
    &~ \qquad \qquad \qquad \qquad \qquad \qquad \qquad \qquad  \beta^{k-1} F(\alpha^{m}x_{\tau \sigma (m+k+1)},\ldots,\alpha^m x_{\tau \sigma (m+k+l)}  ) \big]_{\mathfrak{h}} \\
    & \qquad \qquad +(-1)^k \sum_{\gamma \in Sh (k, l-1)} (-1)^\gamma~  \big[\beta^{l-1} E(\alpha^mx_{\gamma \sigma (m+2)},\ldots,\alpha^m x_{\gamma \sigma (m+k+1)}  ),\\
    &~ \qquad \qquad \qquad \qquad \qquad \qquad \qquad \beta^{k-1} F\big(P(x_{\sigma (1)},\ldots , x_{ \sigma (m+1)}),\alpha^m x_{\gamma \sigma (m+k+2)},\ldots ,\alpha^m x_{\gamma \sigma (m+k+l)}  \big)  \big]_{\mathfrak{h}} \Big)
\end{align*}
    \begin{align*}
    &= \sum_{\sigma \in Sh (m+1,k+ l-1)} (-1)^{\sigma} \sum_{\tau \in Sh (k-1, l)} (-1)^\tau~ \big[ \beta^{l-1} E\big(P(x_{\sigma (1)},\ldots , x_{ \sigma (m+1)}),\alpha^m x_{\tau \sigma (m+2)}, \ldots , \alpha^m x_{\tau \sigma (m+k)}\big),\\
    & \qquad \qquad \qquad \qquad \qquad \qquad \qquad \qquad \qquad \qquad \beta^{k-1} F(\alpha^{m}x_{\tau \sigma (m+k+1)},\ldots,\alpha^m x_{\tau \sigma (m+k+l)}  ) \big]_{\mathfrak{h}} \\
    & \quad  + (-1)^k \sum_{\sigma \in Sh (m+1,k+ l-1)} (-1)^{\sigma} \sum_{\gamma \in Sh (k, l-1)} (-1)^\gamma ~ \big[\beta^{l-1} E(\alpha^m x_{\gamma \sigma (m+2)},
     \ldots,\alpha^m x_{\gamma \sigma (m+k+1)}  ),\\
     & \qquad \qquad \qquad \qquad \qquad \qquad \qquad \beta^{k-1} F\big(P(x_{\sigma (1)},\ldots , x_{ \sigma (m+1)}),\alpha^m x_{\gamma \sigma (m+k+2)},\ldots ,\alpha^m x_{\gamma \sigma (m+k+l)}  \big)  \big]_{\mathfrak{h}}   \\
    &= \sum_{\delta \in Sh (m+1,k-1, l)} (-1)^\delta ~\big[ \beta^{l-1}E\big( P(x_{ \delta(1)}, \ldots, x_{ \delta(m+1)}),\alpha^m x_{ \delta(m+2)},\ldots , \alpha^m x_{\delta (m+k)}  \big),\\
    & \qquad \qquad \qquad \qquad \qquad \qquad \qquad \qquad   \beta^{m+k-1}F(x_{\delta (m+k+1)}, \ldots, x_{\delta (m+k+l)}  )  \big]_{\mathfrak{h}} \\
    & \quad + (-1)^k
    \sum_{\eta \in Sh (m+1,k, l-1)} (-1)^\eta~ \big[ \beta^{l+m-1}E(x_{\eta (m+2)}, \ldots , x_{\eta (m+k+1)}),\\
    & \qquad  \qquad \qquad \qquad \qquad \qquad \qquad \quad \beta^{k-1} F \big(P(x_{\eta (1)},\ldots , x_{ \eta (m+1)}),
    \alpha^m x_{ \eta (k+m+2)}, \ldots ,\alpha^m x_{\eta (k+m+l}  \big) \big]_{\mathfrak{h}}\\
    &=\big( [ \rho (P) E, F]_\mathsf{C}+ (-1)^{mk}[ E,\rho (P)  F]_\mathsf{C} \big)(x_1,\ldots , x_{k+l+m}) \qquad \text{(using (\ref{action-1}) and (\ref{action-2}))}.
\end{align*}
Hence the condition (\ref{action2}) also holds. This proves the result.
\end{proof}

We obtain the following by applying Proposition \ref{semidirect} to the result of Proposition \ref{first-semi-rep}.

\begin{thm}\label{semi-pro-thm}
Let $(\mathfrak{g}, [~,~]_\mathfrak{g}, \alpha)$ and  $(\mathfrak{h}, [~,~]_\mathfrak{h}, \beta)$ be two Hom-Lie algebras. Then the direct sum graded vector space $C^{\bullet +1}_{\mathrm{Hom}} (\mathfrak{g}, \mathfrak{g})  \oplus C^{\bullet}_{\mathrm{Hom}} (\mathfrak{g}, \mathfrak{h})$ carries a graded Lie algebra structure with bracket given by
\begin{align*}
[(P, E), (Q, F)]_\ltimes := \big( [P,Q]_{\mathsf{NR}}~ \! , \! ~  [E,F]_{\mathsf{C}} + i_P F - (-1)^{mn} i_Q E   \big),
\end{align*}
for homogeneous elements $(P, E) \in C^{m +1}_{\mathrm{Hom}} (\mathfrak{g}, \mathfrak{g})  \oplus C^{m}_{\mathrm{Hom}} (\mathfrak{g}, \mathfrak{h})$ and $(Q, F) \in C^{n +1}_{\mathrm{Hom}} (\mathfrak{g}, \mathfrak{g})  \oplus C^{n}_{\mathrm{Hom}} (\mathfrak{g}, \mathfrak{h})$. Moreover, $(C^{\bullet +1 }_\mathrm{Hom} (\mathfrak{g}, \mathfrak{g}), [~,~]_\mathsf{NR})$ is a graded Lie subalgebra and $C^\bullet_\mathrm{Hom} (\mathfrak{g}, \mathfrak{h})$ is an ideal of the graded Lie algebra $( C^{\bullet +1}_{\mathrm{Hom}} (\mathfrak{g}, \mathfrak{g})  \oplus C^{\bullet}_{\mathrm{Hom}} (\mathfrak{g}, \mathfrak{h}), [~,~]_\ltimes).$
\end{thm}


In the rest of this section, we will work with a single Hom-Lie algebra and therefore we will not use any subscript on its Hom-Lie bracket for our convenience. Let  $(\mathfrak{g}, [~,~], \alpha)$ be a Hom-Lie algebra. Consider the semidirect product graded Lie algebra $( C^{\bullet +1 }_\mathrm{Hom} (\mathfrak{g}, \mathfrak{g}) \oplus C^\bullet_\mathrm{Hom} (\mathfrak{g}, \mathfrak{g}) , [~,~]_\ltimes )$ given in Theorem \ref{semi-pro-thm}. Let
\begin{align*}
\mathrm{Gr} ( (-1)^k \delta_{\mathrm{Hom}} ) = \big\{ ( (-1)^k \delta_{\mathrm{Hom}} f, f ) ~ |~ f \in C^k_\mathrm{Hom} (\mathfrak{g}, \mathfrak{g} ) \big\} \subset C^{k +1 }_\mathrm{Hom} (\mathfrak{g}, \mathfrak{g}) \ltimes C^k_\mathrm{Hom} (\mathfrak{g}, \mathfrak{g})
\end{align*}
be the graph of the map $(-1)^k \delta_{\mathrm{Hom}}  : C^k_\mathrm{Hom} (\mathfrak{g}, \mathfrak{g}) \rightarrow C^{k+1}_\mathrm{Hom} (\mathfrak{g}, \mathfrak{g}) $. Define
\begin{align*}
\mathrm{Gr} ((-1)^\bullet \delta_\mathrm{Hom} ) := \bigoplus_k  \mathrm{Gr} ( (-1)^k \delta_{\mathrm{Hom}} ) \subset  C^{\bullet +1 }_\mathrm{Hom} (\mathfrak{g}, \mathfrak{g}) \oplus  C^\bullet_\mathrm{Hom} (\mathfrak{g}, \mathfrak{g}).
\end{align*}

Then we have the following.
\begin{proposition}\label{minus-morphism}
With the above notations, $\mathrm{Gr} ((-1)^\bullet \delta_\mathrm{Hom} ) $ is a graded Lie subalgebra of the semidirect product algebra $ \big( C^{\bullet +1 }_\mathrm{Hom} (\mathfrak{g}, \mathfrak{g}) \oplus C^\bullet_\mathrm{Hom} (\mathfrak{g}, \mathfrak{g}), [~,~]_\ltimes \big)$.
\end{proposition}

\begin{proof}
For $P \in C^m_\mathrm{Hom} (\mathfrak{g}, \mathfrak{g})$ and $Q \in C^n_\mathrm{Hom} (\mathfrak{g}, \mathfrak{g})$, we have
\begin{align*}
&[( (-1)^m \delta_\mathrm{Hom} P, P) , ( (-1)^n \delta_{\mathrm{Hom}} Q, Q) ]_\ltimes \\
&= \big(      (-1)^{m+n} [ \delta_\mathrm{Hom} P, \delta_\mathrm{Hom} Q ]_{\mathsf{NR}} , ~ [P, Q]_\mathsf{C} + (-1)^m ~ i_{(\delta_\mathrm{Hom} P)}Q  - (-1)^{(m+1) n} ~ i_{(\delta_\mathrm{Hom} Q)} P  \big). 
\end{align*}
This is in $\mathrm{Gr} ((-1)^\bullet \delta_\mathrm{Hom})$ if and only if
\begin{align}\label{delta-hom-morphism}
\delta_{\mathrm{Hom}}  \big(  [P, Q]_\mathsf{C} + (-1)^m ~ i_{(\delta_\mathrm{Hom} P)}Q  - (-1)^{(m+1) n} ~ i_{(\delta_\mathrm{Hom} Q)} P   \big) = [\delta_\mathrm{Hom} P, \delta_\mathrm{Hom} Q ]_\mathsf{NR}.
\end{align}
We will now prove this identity. Observe that
\begin{align*}
&\delta_{\mathrm{Hom}}  \big(  [P, Q]_\mathsf{C} + (-1)^m ~ i_{(\delta_\mathrm{Hom} P)}Q  - (-1)^{(m+1) n} ~ i_{(\delta_\mathrm{Hom} Q)} P  \big) \\
&= \delta_\mathrm{Hom}    (  i_P (\delta_\mathrm{Hom} Q)  ) + (-1)^{m-1} \cancel{\delta_\mathrm{Hom} ( i_{(\delta_\mathrm{Hom} P)} Q )} + (-1)^m \cancel{\delta_\mathrm{Hom} ( i_{(\delta_\mathrm{Hom} P)} Q ) } - (-1)^{(m+1)n } \delta_\mathrm{Hom} ( i_{(\delta_\mathrm{Hom} Q)} P  ) \\
 & \qquad \qquad \qquad \qquad \qquad \qquad \qquad \qquad \qquad \qquad  (\text{by Proposition } \ref{cup-new-new}) \\
 &= \delta_\mathrm{Hom}  \big( i_P (\delta_\mathrm{Hom} Q)     - (-1)^{(m+1)n } i_{(\delta_\mathrm{Hom} Q)} P   \big) \\
 &= \delta_\mathrm{Hom} [ P, \delta_\mathrm{Hom} Q ]_\mathsf{NR}\\
 &= [\delta_\mathrm{Hom} P, \delta_\mathrm{Hom} Q]_\mathsf{NR}.
\end{align*}
Hence we have proved the result.
\end{proof}

The above proposition suggests us to consider the following definition.

\begin{definition}
Let $(\mathfrak{g}, [~,~], \alpha)$ be a Hom-Lie algebra. For each $m, n \geq 1$, we define a bracket $$[~,~]_{\mathsf{FN}} : C^m_{\mathrm{Hom}} (\mathfrak{g}, \mathfrak{g}) \times C^n_{\mathrm{Hom}} (\mathfrak{g}, \mathfrak{g}) \rightarrow C^{m+n}_{\mathrm{Hom}} (\mathfrak{g}, \mathfrak{g})$$ by
\begin{align}\label{fnn}
[P,Q]_{\mathsf{FN}} := [P, Q]_\mathsf{C} + (-1)^m ~ i_{(\delta_\mathrm{Hom} P)}Q  - (-1)^{(m+1) n} ~ i_{(\delta_\mathrm{Hom} Q)} P ,
\end{align}
for $P \in  C^m_{\mathrm{Hom}} (\mathfrak{g}, \mathfrak{g})$ and $Q \in C^n_{\mathrm{Hom}} (\mathfrak{g}, \mathfrak{g})$. The bracket $[~,~]_{\mathsf{FN}}$ is called the \textbf{Fr\"olicher-Nijenhuis bracket}.
\end{definition}

By using Proposition \ref{cup-new-new}, the Fr\"olicher-Nijenhuis bracket $[~,~]_\mathsf{FN}$ can be expressed as
\begin{align}
[P,Q]_\mathsf{FN} =~& i_P (\delta_\mathrm{Hom} Q ) + (-1)^{m-1} \cancel{ i_{(\delta_\mathrm{Hom} P)} Q} + (-1)^m \delta_\mathrm{Hom} ( i_P Q) \nonumber \\ 
& \qquad \qquad  + (-1)^m ~ \cancel{i_{(\delta_\mathrm{Hom} P)} Q} - (-1)^{(m+1)n}  i_{(\delta_\mathrm{Hom} Q)} P \nonumber\\
=~& [ P, \delta_\mathrm{Hom} Q ]_\mathsf{NR}  + (-1)^m \delta_\mathrm{Hom} ( i_P Q ).
\end{align}
Moreover, using the explicit expression of the cup product bracket $[~,~]_\mathsf{C}$ and the contraction operator $i$ in the defining identity (\ref{fnn}), we get that
\begin{align}\label{explicit-fn}
&[P,Q]_\mathsf{FN} (x_1, \ldots, x_{m+n}) \\
&= \sum_{\sigma \in Sh (m,n)} (-1)^\sigma ~[\alpha^{n-1} P (x_{\sigma (1)}, \ldots, x_{\sigma (m)} ), \alpha^{m-1} Q ( x_{\sigma (m+1)}, \ldots, x_{\sigma (m+n)} ) ]_\mathfrak{g} \nonumber \\
&+ (-1)^m \sum_{\sigma \in Sh (m+1, n-1)} (-1)^\sigma ~ Q \big( (\delta_\mathrm{Hom} P) ( x_{\sigma (1)}, \ldots, x_{\sigma (m +1)} ), \alpha^m x_{\sigma (m+2)}, \ldots, \alpha^m x_{\sigma (m+n)} \big) \nonumber \\
&- (-1)^{(m+1)n }   \sum_{\sigma \in Sh (n+1, m-1)} (-1)^\sigma ~ P \big( (\delta_\mathrm{Hom} Q) ( x_{\sigma (1)}, \ldots, x_{\sigma (n +1)} ), \alpha^n x_{\sigma (n+2)}, \ldots, \alpha^n x_{\sigma (m+n)} \big), \nonumber 
\end{align}
for $P \in C^m_\mathrm{Hom} (\mathfrak{g}, \mathfrak{g})$, $Q \in C^n_\mathrm{Hom} (\mathfrak{g}, \mathfrak{g})$ and $x_1, \ldots, x_{m+n} \in \mathfrak{g}$. Note that an implicit form of the Fr\"{o}licher-Nijenhuis bracket was defined in \cite{das-sen}. Here our bracket is explicit and it is also useful to connect with the Nijenhuis-Richardson graded Lie algebra by the following result.

\begin{thm}
Let $(\mathfrak{g}, [~,~], \alpha)$ be a Hom-Lie algebra. 

(i) Then the Fr\"olicher-Nijenhuis bracket makes $(C^\bullet_{\mathrm{Hom}} (\mathfrak{g}, \mathfrak{g}) , [~,~]_{\mathsf{FN}} )$ into a graded Lie algebra, called the Fr\"olicher-Nijenhuis algebra. 

(ii) Moreover, the Hom-Lie algebra differential $\delta_{\mathrm{Hom}} : C^\bullet_{\mathrm{Hom}} (\mathfrak{g}, \mathfrak{g}) \rightarrow C^{\bullet +1 }_{\mathrm{Hom}} (\mathfrak{g}, \mathfrak{g})$ is a graded Lie algebra morphism from the Fr\"olicher-Nijenhuis algebra $(C^\bullet_{\mathrm{Hom}} (\mathfrak{g}, \mathfrak{g}) , [~,~]_{\mathsf{FN}} )$ to the Nijenhuis-Richardson algebra $(C^{\bullet +1}_{\mathrm{Hom}} (\mathfrak{g}, \mathfrak{g}) , [~,~]_{\mathrm{NR}})$.
\end{thm}

\begin{proof}
(i) Consider the injective map 
$\Theta := ((-1)^\bullet \delta_\mathrm{Hom}, \mathrm{id} ) : C^\bullet_\mathrm{Hom} (\mathfrak{g}, \mathfrak{g}) \rightarrow   C^{\bullet +1}_\mathrm{Hom} (\mathfrak{g}, \mathfrak{g}) \oplus C^\bullet_\mathrm{Hom} (\mathfrak{g}, \mathfrak{g})$
of graded vector spaces. It follows from Proposition \ref{minus-morphism} that 
\begin{align*}
\Theta :  ( C^\bullet_\mathrm{Hom} (\mathfrak{g}, \mathfrak{g}) , [~,~]_\mathsf{FN}) \rightarrow (  C^{\bullet +1}_\mathrm{Hom} (\mathfrak{g}, \mathfrak{g}) \oplus C^\bullet_\mathrm{Hom} (\mathfrak{g}, \mathfrak{g}) , [~,~]_\ltimes)
\end{align*}
preserves the corresponding brackets. This implies that $( C^\bullet_\mathrm{Hom} (\mathfrak{g}, \mathfrak{g}) , [~,~]_\mathsf{FN}) $ is a graded Lie algebra.

(ii) This part follows from the definition of the Fr\"olicher-Nijenhuis bracket and the identity (\ref{delta-hom-morphism}).
\end{proof}


\medskip

\noindent {\bf Nijenhuis operators on Hom-Lie algebras.} Let $(\mathfrak{g}, [~,~], \alpha)$ be a Hom-Lie algebra and $N : \mathfrak{g} \rightarrow \mathfrak{g}$ be a linear map satisfying $\alpha \circ N = N \circ \alpha$. Thus, $N \in C^1_\mathrm{Hom} (\mathfrak{g}, \mathfrak{g})$. We define a bilinear skew-symmetric bracket 
\begin{align*}
[x, y]^N := [Nx, y] + [x, Ny] - N [x, y],  \text{ for } x, y \in \mathfrak{g}.
\end{align*}
Then we have $\alpha ([x, y]^N) = [\alpha (x) , \alpha (y)]^N$ for $x, y \in \mathfrak{g}$.
Note that the bracket $[~,~]^N$ is trivial if $N$ is a Hom-derivation for the given Hom-Lie bracket $[~,~]$. Therefore, $[~,~]^N$ measures the Hom-derivation property of $N$ on the given Hom-Lie algebra. The following result gives a necessary and sufficient condition under which the bracket $[~,~]^N$ defines a new Hom-Lie algebra structure on $\mathfrak{g}.$

\begin{proposition}\label{deformed-prop}
Let $(\mathfrak{g}, [~,~], \alpha)$  be a Hom-Lie algebra and $N : \mathfrak{g} \rightarrow \mathfrak{g}$ be a linear map satisfying $\alpha \circ N = N \circ \alpha$. Then $(\mathfrak{g}, [~,~]^N, \alpha)$ is a Hom-Lie algebra if and only if $\delta_\mathrm{Hom} ([ N,N]_\mathsf{FN} ) = 0$.
\end{proposition}

\begin{proof}
From the definition of $\delta_\mathrm{Hom}$, we observe that
\begin{align*}
(\delta_\mathrm{Hom} N)(x, y) = [x, Ny] + [Nx, y] - N [x,y] = [x, y]^N, ~ \text{ for } x, y \in \mathfrak{g}.
\end{align*}
On the other hand, $[~,~]^N$ is a Hom-Lie bracket if and only if $[ \delta_\mathrm{Hom} N, \delta_\mathrm{Hom} N ]_\mathsf{NR} = 0$. However, we have
\begin{align*}
\delta_\mathrm{Hom} ([N, N]_\mathsf{FN}) = [  \delta_\mathrm{Hom} N, \delta_\mathrm{Hom} N ]_\mathsf{NR} .
\end{align*}
This implies that $[~,~]^N$ is a Hom-Lie bracket if and only if $\delta_\mathrm{Hom} ([N, N]_\mathsf{FN}) = 0$.
\end{proof}

Note that, for a linear map $N : \mathfrak{g} \rightarrow \mathfrak{g}$ satisfying $\alpha \circ N = N \circ \alpha$ (equivalently, $N \in C^1_\mathrm{Hom} (\mathfrak{g}, \mathfrak{g})$), we have from (\ref{explicit-fn}) that
\begin{align}\label{nn-fn}
[N, N ]_\mathsf{FN} (x, y) =~& 2 \big(    [Nx, Ny] - N ( [Nx, y] + [x, Ny] - N[x, y]) \big) \nonumber \\
=~& 2 \big(   [Nx, Ny] - N [x, y]^N \big).
\end{align}

\begin{definition} Let $(\mathfrak{g}, [~,~], \alpha)$ be a Hom-Lie algebra.
A linear map $N : \mathfrak{g} \rightarrow \mathfrak{g}$ satisfying $\alpha \circ N = N \circ \alpha$ is said to be a {\bf Nijenhuis operator} on the Hom-Lie algebra if
\begin{align}\label{nij-defn}
[Nx, Ny] = N ([x, y]^N),~ \text{ for } x, y \in \mathfrak{g}.
\end{align}
\end{definition}

Note that the condition (\ref{nij-defn}) is equivalent to say that $[N, N]_\mathsf{FN} = 0$, i.e. $N$ is a Maurer-Cartan element in the Fr\"{o}licher-Nijenhuis algebra. Nijenhuis operators on Hom-Lie algebras are studied in \cite{das-sen,LCM}. The following result has been proven there which can be seen as a consequence of our Proposition \ref{deformed-prop} and the identity (\ref{nn-fn}).

\begin{proposition}
Let $(\mathfrak{g}, [~,~], \alpha)$ be a Hom-Lie algebra and $N : \mathfrak{g} \rightarrow \mathfrak{g}$ be a Nijenhuis operator on it. Then $(\mathfrak{g}, [~,~]^N, \alpha)$ is a Hom-Lie algebra and $N : (\mathfrak{g}, [~,~]^N, \alpha) \rightarrow (\mathfrak{g}, [~,~], \alpha)$ is a morphism of Hom-Lie algebras.
\end{proposition}

The Hom-Lie algebra $(\mathfrak{g}, [~,~]^N, \alpha)$ in the above proposition is said to be the deformed Hom-Lie algebra induced by the Nijenhuis operator $N$. This deformed Hom-Lie algebra $(\mathfrak{g}, [~,~]^N, \alpha)$ is also compatible with the given Hom-Lie algebra $(\mathfrak{g}, [~,~], \alpha)$ in the sense that, for all $t \in {\bf k}$, the triple $(\mathfrak{g}, [~,~] + t [~,~]^N, \alpha)$ is a Hom-Lie algebra \cite{das-comp}. Since $[~,~]^N = \delta_\mathrm{Hom} N$, we have that $(\mathfrak{g}, [~,~] + t \delta_\mathrm{Hom} N, \alpha)$ is a Hom-Lie algebra. In other words, the Nijenhuis operator $N$ induces a linear deformation of the given Hom-Lie algebra $(\mathfrak{g}, [~,~], \alpha)$. This result has been studied in \cite{LCM}.

We have seen that a Nijenhuis operator $N$ on a Hom-Lie algebra $(\mathfrak{g}, [~,~], \alpha)$ can be seen as a Maurer-Cartan element in the graded Lie algebra $(C^\bullet_\mathrm{Hom}(\mathfrak{g}, \mathfrak{g}), [~,~]_\mathsf{FN} )$. This Maurer-Cartan characterization of a Nijenhuis operator allows one to study cohomology and deformations of a Nijenhuis operator (similar to the study of Hom-Lie algebra morphisms given in Section \ref{sec3}).

\medskip

\noindent {\bf A matched pair between Nijenhuis-Richardson algebra and Fr\"{o}licher-Nijenhuis algebra.} Here we will show that the Nijenhuis-Richardson algebra and the Fr\"{o}licher-Nijenhuis algebra form a matched pair of graded Lie algebras. We start with the following definition.

\begin{definition}
A {\bf matched pair of graded Lie algebras} consists of a quadruple 
\begin{align*}
\big( (\mathcal{A}, [~,~]_\mathcal{A}), (\mathcal{B}, [~, ~]_\mathcal{B}), \rho, \psi \big)
\end{align*}
of two graded Lie algebras $ (\mathcal{A}, [~,~]_\mathcal{A})$ and $ (\mathcal{B}, [~, ~]_\mathcal{B})$ together with two degree $0$ bilinear maps $\rho : \mathcal{A} \times \mathcal{B} \rightarrow \mathcal{B}, ~(a,b) \mapsto \rho (a) (b)$ and $\psi : \mathcal{B} \times \mathcal{A} \rightarrow \mathcal{A}, ~ (b,a) \mapsto \psi (b) (a)$ satisfying
\begin{align*}
\rho ([a, a']_\mathcal{A})(b) =~& \rho (a) \rho (a')(b) - (-1)^{|a||a'|}  \rho (a') \rho (a)(b),\\
\rho (a) ([b, b']_\mathcal{B} ) =~& [ \rho (a) b, b']_\mathcal{B} + (-1)^{|a||b|} [b, \rho (a) b']_\mathcal{B}  +    (-1)^{(|a| + |b|)|b'|}   \rho (\psi (b') a) b -   (-1)^{|a||b|}  \rho ( \psi (b) a) b',\\
\psi ([b, b']_\mathcal{B})(a) =~& \psi (b) \psi (b')(a) - (-1)^{|b||b'|}  \psi (b') \psi (b)(a),\\
\psi (b) ([a, a']_\mathcal{A}) =~&  [ \psi (b) a, a']_\mathcal{A} + (-1)^{|b||a|} [a, \psi (b) a']_\mathcal{A}  +   (-1)^{(|b| + |a|)|a'|}   \psi (\rho (a') b) a -  (-1)^{|b||a|}  \psi ( \rho (a) b) a',
\end{align*}
for all homogeneous elements $a, a' \in \mathcal{A}$ and $b, b' \in \mathcal{B}$.
\end{definition}

When $\psi = 0$ (resp. $\rho = 0$), it follows from the above definition that $\rho$ (resp. $\psi$) defines an action of the graded Lie algebra $\mathcal{A}$ on $\mathcal{B}$ (resp. $\mathcal{B}$ on $\mathcal{A}$). The above definition is a graded version of the notion of a matched pair of Lie algebras \cite{koss}. Similar to a matched pair of Lie algebras, we have the following result in the graded context.

\begin{proposition}\label{match-pro} Let $( (\mathcal{A}, [~,~]_\mathcal{A}), (\mathcal{B}, [~,~]_\mathcal{B}), \rho, \psi)$ be a matched pair of graded Lie algebras. Then the direct sum $\mathcal{A} \oplus \mathcal{B}$ carries a graded Lie algebra structure with bracket
\begin{align}\label{match-pro-for}
[(a, b), (a', b')]_{\bowtie} :=  \big(  [a, a']_\mathcal{A} + \psi (b) a' - (-1)^{|b||a'|} \psi (b')a ~ \! , \! ~ [b, b']_\mathcal{B} + \rho (a) b' - (-1)^{|b||a'|} \rho (a')b \big),
\end{align}
for homogeneous elements $(a, b), (a', b') \in \mathcal{A} \oplus \mathcal{B}$.
\end{proposition}

 The graded Lie algebra constructed in the above proposition is called the {\em bicrossed product} (also called the {\em matched product}), and it is denoted by $(\mathcal{A}, [~,~]_\mathcal{A}) \bowtie (\mathcal{B}, [~,~]_\mathcal{B})$ or simply by $\mathcal{A} \bowtie \mathcal{B}$. 
 It follows from the bracket (\ref{match-pro-for}) that both $(\mathcal{A}, [~,~]_\mathcal{A})$ and $(\mathcal{B}, [~,~]_\mathcal{B})$ are graded Lie subalgebras of $\mathcal{A} \bowtie \mathcal{B}$.

 \begin{remark}\label{remark-matched}
     Let $(\mathcal{A}, [~,~]_\mathcal{A})$ and $(\mathcal{B}, [~,~]_\mathcal{B})$ be two graded Lie algebras. Suppose there is some graded Lie bracket $\llbracket ~,~ \rrbracket$ on the direct sum $\mathcal{A} \oplus \mathcal{B}$ for which $(\mathcal{A}, [~,~]_\mathcal{A})$ and $(\mathcal{B}, [~,~]_\mathcal{B})$ are both graded Lie subalgebras of $(\mathcal{A} \oplus \mathcal{B}, \llbracket ~,~ \rrbracket)$. Then there are unique degree $0$ bilinear maps $\rho : \mathcal{A} \times \mathcal{B} \rightarrow \mathcal{B}$ and $\psi : \mathcal{B} \times \mathcal{A} \rightarrow \mathcal{A}$ that makes $\big(  (\mathcal{A}, [~,~]_\mathcal{A}), (\mathcal{B}, [~,~]_\mathcal{B}), \rho, \psi  \big)$ into a matched pair of graded Lie algebras such that the associated bicrossed product coincides with $(\mathcal{A} \oplus \mathcal{B} , \llbracket ~,~ \rrbracket)$. The maps $\rho$ and $\psi$ are given by
     \begin{align*}
         \rho (a) b :=  \mathrm{pr}_2 \llbracket (a,0), (0, b) \rrbracket   ~~~~ \text{ and } ~~~~ \psi (b) a := \mathrm{pr}_1 \llbracket (0, b) , (a, 0) \rrbracket, \text{ for } a \in \mathcal{A}, b \in \mathcal{B}.
     \end{align*}
     Here $\mathrm{pr}_1 : \mathcal{A} \oplus \mathcal{B} \rightarrow \mathcal{A}$ and $\mathrm{pr}_2 : \mathcal{A} \oplus \mathcal{B} \rightarrow \mathcal{B}$ are the projections.
 \end{remark}


Let $(\mathfrak{g}, [~,~], \alpha)$ be a Hom-Lie algebra. Consider the graded vector space $C^{\bullet +1 }_{\mathrm{Hom}} (\mathfrak{g}, \mathfrak{g}) \oplus C^{\bullet }_{\mathrm{Hom}} (\mathfrak{g}, \mathfrak{g})$ and define a bilinear graded skew-symmetric bracket on it by
\begin{align}\label{mmaatt}
\llbracket (P, E), (Q, F) \rrbracket :=  \big(   [P, Q]_\mathsf{NR}  +  [ E, Q]_\mathsf{FN} - (-1)^{mn}   [F,P]_\mathsf{FN}    ~ ,~ [E, F]_\mathsf{FN} + i_P F - (-1)^{mn} i_Q E   \big),
\end{align}
for $(P,E) \in C^{m+1}_\mathrm{Hom} (\mathfrak{g}, \mathfrak{g}) \oplus C^m_\mathrm{Hom} (\mathfrak{g}, \mathfrak{g})$ and $(Q, F) \in C^{n+1}_\mathrm{Hom} (\mathfrak{g}, \mathfrak{g})  \oplus C^n_\mathrm{Hom} (\mathfrak{g}, \mathfrak{g})$. Then we have the following result.

\begin{proposition}
    Let $(\mathfrak{g}, [~,~], \alpha)$ be a Hom-Lie algebra. Then $\big( C^{\bullet +1 }_{\mathrm{Hom}} (\mathfrak{g}, \mathfrak{g}) \oplus C^{\bullet }_{\mathrm{Hom}} (\mathfrak{g}, \mathfrak{g})  , \llbracket ~, ~ \rrbracket \big)$ is a graded Lie algebra.
\end{proposition}

\begin{proof}
    We consider the degree $0$ map $\Psi :  C^{\bullet +1 }_{\mathrm{Hom}} (\mathfrak{g}, \mathfrak{g}) \oplus C^{\bullet }_{\mathrm{Hom}} (\mathfrak{g}, \mathfrak{g}) \rightarrow  C^{\bullet +1 }_{\mathrm{Hom}} (\mathfrak{g}, \mathfrak{g}) \oplus C^{\bullet }_{\mathrm{Hom}} (\mathfrak{g}, \mathfrak{g})$ given by
    \begin{align*}
        \Psi (P, E) := (P + (-1)^m \delta_{\mathrm{Hom}} E, E), ~ \text{ for } (P, E) \in C^{m +1 }_{\mathrm{Hom}} (\mathfrak{g}, \mathfrak{g}) \oplus  C^{m}_{\mathrm{Hom}} (\mathfrak{g}, \mathfrak{g}).
    \end{align*}
    We claim that $\Psi : \big(   C^{\bullet +1 }_{\mathrm{Hom}} (\mathfrak{g}, \mathfrak{g}) \oplus C^{\bullet }_{\mathrm{Hom}} (\mathfrak{g}, \mathfrak{g}), \llbracket ~, ~ \rrbracket \big) \rightarrow \big(   C^{\bullet +1 }_{\mathrm{Hom}} (\mathfrak{g}, \mathfrak{g}) \oplus C^{\bullet }_{\mathrm{Hom}} (\mathfrak{g}, \mathfrak{g}), [~,~]_\ltimes  \big)$ preserves the brackets. For any $(P, E) \in C^{m +1 }_{\mathrm{Hom}} (\mathfrak{g}, \mathfrak{g}) \oplus  C^{m}_{\mathrm{Hom}} (\mathfrak{g}, \mathfrak{g})$ and $(Q, F) \in  C^{n +1 }_{\mathrm{Hom}} (\mathfrak{g}, \mathfrak{g}) \oplus  C^{n}_{\mathrm{Hom}} (\mathfrak{g}, \mathfrak{g})$, we observe that
    \begin{align}\label{iso-1-1}
&\Psi \big(    \llbracket (P,E), (Q,F) \rrbracket \big)  \nonumber \\
&= \big(    [P,Q]_\mathsf{NR} + [E,Q]_\mathsf{FN} - (-1)^{mn} [F,P]_\mathsf{FN} + (-1)^{m+n} ~ \delta_\mathrm{Hom}  ( [E,F]_\mathsf{FN} + i_P F  - (-1)^{mn} i_Q E ),   \nonumber \\
& \qquad [E,F]_\mathsf{FN} + i_P F  - (-1)^{mn} i_Q E  \big)  \nonumber  \\
&= \big(  [P,Q]_\mathsf{NR}  - (-1)^{m (n+1)} [Q, E]_\mathsf{FN} + (-1)^{mn + (m+1) n } [P,F]_\mathsf{FN} + (-1)^{m+n} \delta_\mathrm{Hom} [E,F]_\mathsf{FN}   \nonumber \\
& \qquad + (-1)^{m+n} \delta_\mathrm{Hom} ( i_P F ) - (-1)^{m+n+mn} \delta_\mathrm{Hom} (i_Q E),  \nonumber  \\
& \qquad \quad [E,F]_\mathsf{C} + (-1)^m i_{(\delta_\mathrm{Hom} E)} F  - (-1)^{(m+1) n } i_{(\delta_\mathrm{Hom} F)} E + i_P F - (-1)^{mn} i_Q E  \big)  \nonumber \\
&=  \big(    [P,Q]_\mathsf{NR}  -(-1)^{m(n+1)} [Q, \delta_\mathrm{Hom} E ]_\mathrm{NR} + (-1)^n [P, \delta_\mathrm{Hom} F]_\mathsf{NR} + (-1)^{m+n} [ \delta_\mathrm{Hom} E, \delta_\mathrm{Hom} F]_\mathsf{NR}, \\
& \qquad \quad [E,F]_\mathsf{C} + i_{(\delta_\mathrm{Hom} E)}(-1)^m F  - (-1)^{(m+1) n }i_{(\delta_\mathrm{Hom} F)} E   + i_P F  - (-1)^{mn} i_Q E  \big).  \nonumber 
\end{align}
On the other hand, 
\begin{align}\label{iso-2-2}
&[ \Psi (P, E), \Psi (Q, F) ]_\ltimes  \nonumber  \\
&= [ ( P + (-1)^m \delta_\mathrm{Hom} E, E), ( Q + (-1)^n \delta_\mathrm{Hom} F, F ) ]_\ltimes  \nonumber \\
&= \big(    [P, Q]_\mathsf{NR} + (-1)^n [P, \delta_\mathrm{Hom} F ]_\mathsf{NR} + (-1)^m [\delta_\mathrm{Hom} E, Q]_\mathsf{NR} + (-1)^{m+n} [\delta_\mathrm{Hom} E, \delta_\mathrm{Hom} F ]_\mathsf{NR} , \\
&\qquad \quad i_P F + (-1)^m i_{(\delta_\mathrm{Hom} E )} F - (-1)^{mn} i_Q E  - (-1)^{mn + n} i_{(\delta_\mathrm{Hom} F)} E + [E, F]_\mathsf{C} \big). \nonumber 
\end{align}
It follows from the expressions in (\ref{iso-1-1}) and (\ref{iso-2-2}) that
$\Psi \big(   [(P,E), (Q,F)]_{\bowtie} \big) = [ \Psi (P, E), \Psi (Q, F) ]_\ltimes.$ This proves our claim. Since $\Psi$ is an isomorphism of graded vector spaces, the graded skew-symmetry and the graded Jacobi identity of the bracket $\llbracket ~, ~ \rrbracket$ follows from the corresponding properties of the bracket $[~,~]_\ltimes$.
\end{proof}

It follows from (\ref{mmaatt}) that the Nijenhuis-Richardson algebra $(  C^{\bullet +1 }_{\mathrm{Hom}} (\mathfrak{g}, \mathfrak{g}), [~,~]_{\mathsf{NR}} )$ and the Fr\"{o}licher-Nijenhuis algebra $( C^\bullet_{\mathrm{Hom}} (\mathfrak{g}, \mathfrak{g}), [~,~]_{\mathsf{FN}})$ are both graded Lie subalgebras of $\big(  C^{\bullet +1 }_{\mathrm{Hom}} (\mathfrak{g}, \mathfrak{g}) \oplus C^{\bullet }_{\mathrm{Hom}} (\mathfrak{g}, \mathfrak{g})  , \llbracket ~, ~ \rrbracket \big)$. Hence by Remark \ref{remark-matched}, we get the following.

\begin{thm}
    Let $(\mathfrak{g}, [~,~], \alpha)$ be a Hom-Lie algebra. Define degree $0$ bilinear maps
\begin{align*}
&\rho : C^{\bullet +1 }_{\mathrm{Hom}} (\mathfrak{g}, \mathfrak{g}) \times C^{\bullet }_{\mathrm{Hom}} (\mathfrak{g}, \mathfrak{g}) \rightarrow C^{\bullet}_{\mathrm{Hom}} (\mathfrak{g}, \mathfrak{g}),~ \rho (P)(E) := i_P E ,\\
&\psi : C^{\bullet}_{\mathrm{Hom}} (\mathfrak{g}, \mathfrak{g}) \times C^{\bullet +1 }_{\mathrm{Hom}} (\mathfrak{g}, \mathfrak{g}) \rightarrow C^{\bullet +1 }_{\mathrm{Hom}} (\mathfrak{g}, \mathfrak{g}),~\psi (E) (P) := [E, P]_{\mathsf{FN}}.
\end{align*} 
Then the quadruple $\big(  (  C^{\bullet +1 }_{\mathrm{Hom}} (\mathfrak{g}, \mathfrak{g}), [~,~]_{\mathsf{NR}} ),  ( C^\bullet_{\mathrm{Hom}} (\mathfrak{g}, \mathfrak{g}), [~,~]_{\mathsf{FN}}), \rho, \psi \big)$ is a matched pair of graded Lie algebras. Moreover, the corresponding bicrossed product bracket coincides with (\ref{mmaatt}).
\end{thm}

\section{Derived bracket for Hom-Lie algebras}\label{sec5}

In this section, we define a new graded Lie bracket, which we refer as the derived bracket on the space of cochains of a Hom-Lie algebra. For any fixed $\lambda \in {\bf k}$, there is a suitable differential that makes the derived algebra into a differential graded Lie algebra. Its Maurer-Cartan elements are precisely Rota-Baxter operators of weight $\lambda$. Finally, we generalize the above differential graded Lie algebra to deal with relative Rota-Baxter operators.

Let $(\mathfrak{g}, [~,~], \alpha)$ be a Hom-Lie algebra. Consider the  graded Lie algebra $( C^{\bullet +1 }_\mathrm{Hom} (\mathfrak{g}, \mathfrak{g}) \oplus C^{\bullet }_\mathrm{Hom} (\mathfrak{g}, \mathfrak{g}), [~, ~]_\ltimes )$ given in Theorem \ref{semi-pro-thm}.
Let $\theta : C^n_{\mathrm{Hom}} (\mathfrak{g}, \mathfrak{g}) \rightarrow  C^{n +1 }_{\mathrm{Hom}} (\mathfrak{g}, \mathfrak{g})$ be the map given in (\ref{theta}) and we define
\begin{align*}
\mathrm{Gr} (\theta) := \bigoplus_n \{ (\theta P, P) |~ P \in C^n_\mathrm{Hom}(\mathfrak{g}, \mathfrak{g}) \} \subset C^{\bullet +1 }_\mathrm{Hom} (\mathfrak{g}, \mathfrak{g}) \oplus C^{\bullet }_\mathrm{Hom} (\mathfrak{g}, \mathfrak{g}).
\end{align*}
Then we have the following result.

\begin{proposition}\label{ks-prop}
With the above notations, $\mathrm{Gr }(\theta)$ is a graded Lie subalgebra of the semidirect product algebra $(C^{\bullet +1 }_{\mathrm{Hom}} (\mathfrak{g}, \mathfrak{g}) \oplus C^{\bullet}_{\mathrm{Hom}} (\mathfrak{g}, \mathfrak{g}) , [~,~]_\ltimes )$.
\end{proposition}

\begin{proof}
For $P \in C^m_\mathrm{Hom} (\mathfrak{g}, \mathfrak{g})$ and $Q \in C^n_\mathrm{Hom} (\mathfrak{g}, \mathfrak{g})$, we have
\begin{align*}
[(\theta P, P), (\theta Q,Q)]_\ltimes 
= \big(  [\theta P, \theta Q]_\mathsf{NR}, ~ [P, Q]_\mathsf{C} + i_{\theta P} Q -(-1)^{mn}i_{\theta Q} P ).
\end{align*}
This is in $\mathrm{Gr} (\theta)$ if and only if 
\begin{align}\label{theta-mor}
\theta \big( [P, Q]_\mathsf{C} + i_{\theta P} Q -(-1)^{mn}i_{\theta Q} P\big) = [\theta P, \theta Q]_\mathsf{NR}.
\end{align}
We will now prove this identity. For this, we observe that
\begin{align*}
&\theta \big( [P, Q]_\mathsf{C} + i_{\theta P} Q -(-1)^{mn}i_{\theta Q} P\big)  \\
&= (-1)^n [\theta P, Q]_\mathsf{C} + [P, \theta Q]_\mathsf{C} + \theta ( i_{\theta P}Q ) - (-1)^{mn} \theta (i_{\theta Q}P ) \qquad (\text{by Proposition }\ref{tr-theta}) \\
& = i_{\theta P} \theta Q  - \cancel{\theta (i_{\theta P} Q)} - (-1)^{m(n+1)} (-1)^m i_{\theta Q}\theta P   + (-1)^{m(n+1)} (-1)^m \cancel{\theta (i_{\theta Q} P )}  \\
& \qquad  \qquad \qquad \qquad \qquad \qquad  + \cancel{\theta (i_{\theta P} Q )} - (-1)^{mn} \cancel{\theta (i_{\theta Q} P )}  \qquad (\text{by Lemma } \ref{lemma-lemma}) \\
&=i_{\theta P} \theta Q   - (-1)^{mn}i_{\theta Q} \theta P \\
&= [\theta P, \theta Q]_\mathsf{NR}.
\end{align*}
Hence we have proved the result.
\end{proof}

The above proposition suggests us to introduce the following definition.

\begin{definition}
Let $(\mathfrak{g}, [~,~], \alpha)$ be a Hom-Lie algebra. For each $m, n \geq 1$, we define a degree $0$ bilinear bracket
$[~,~]_{\mathsf{D}} : C^m_{\mathrm{Hom}} (\mathfrak{g}, \mathfrak{g}) \times C^n_{\mathrm{Hom}} (\mathfrak{g}, \mathfrak{g}) \rightarrow C^{m+n}_{\mathrm{Hom}} (\mathfrak{g}, \mathfrak{g})$ by
\begin{align}\label{ks-first}
[P,Q]_{\mathsf{D}} : = [P, Q ]_{\mathsf{C}} +i_{\theta P} Q - (-1)^{mn} i_{\theta Q} P,
\end{align}
for $P \in C^m_\mathrm{Hom} (\mathfrak{g}, \mathfrak{g})$ and $Q \in C^n_\mathrm{Hom} (\mathfrak{g}, \mathfrak{g})$. 
The bracket $[~,~]_{\mathsf{D}}$ is called the {\bf derived bracket}.
\end{definition}

From the definition of the cup product bracket and the explicit form of $\theta$ yields that
\begin{align}
    &[P, Q]_\mathsf{D}(x_1, x_2, \ldots , x_{m+n}) \\
    &~= \sum_{\sigma \in Sh (m,n)} (-1)^\sigma [\alpha^{n-1}P(x_{\sigma (1)}, \ldots , x_{\sigma (m)}), \alpha^{m-1} Q(x_{\sigma (m+1)},\ldots , x_{\sigma (m+n)})] \nonumber\\ \nonumber
    &~~- \sum_{\sigma \in Sh (m,1,n-1)} (-1)^\sigma Q\big([P(x_{\sigma (1)}, \ldots , x_{\sigma (m)}), \alpha^{m-1}x_{\sigma (m+1)}], \alpha^m x_{\sigma (m+2)},\ldots , \alpha^m x_{\sigma (m+n)}  \big)\\ 
    &~~+ (-1)^{mn} \sum_{\sigma \in Sh (n,1,m-1)} (-1)^\sigma P\big([Q(x_{\sigma (1)}, \ldots , x_{\sigma (n)}), \alpha^{n-1}x_{\sigma (n+1)}], \alpha^n x_{\sigma (n+2)},\ldots , \alpha^n x_{\sigma (m+n)}  \big), \nonumber
\end{align}
for $P \in C^m_\mathrm{Hom} (\mathfrak{g}, \mathfrak{g})$, $Q \in C^n_\mathrm{Hom} (\mathfrak{g}, \mathfrak{g})$ and $x_1, \ldots, x_{m+n} \in \mathfrak{g}$. This explicit description shows that the bracket $[~,~]_\mathsf{D}$ coincides with the derived bracket considered in \cite{mishra-naolekar}. However, our description of the bracket is useful to connect it with the Nijenhuis-Richardson graded Lie algebra. More precisely, we have the following result.

\begin{thm}\label{Kodai-thm}
Let $(\mathfrak{g}, [~,~], \alpha)$ be a Hom-Lie algebra.

(i) Then the derived bracket makes $(C^\bullet_{\mathrm{Hom}} (\mathfrak{g}, \mathfrak{g}), [~,~]_\mathsf{D})$ into a graded Lie algebra, called the derived algebra.

(ii) Moreover, the map $\theta : C^\bullet_{\mathrm{Hom}} (\mathfrak{g}, \mathfrak{g}) \rightarrow C^{\bullet + 1 }_{\mathrm{Hom}} (\mathfrak{g}, \mathfrak{g})$ is a morphism of graded Lie algebras from $(C^\bullet_{\mathrm{Hom}} (\mathfrak{g}, \mathfrak{g}), [~,~]_{\mathsf{D}} )$ to  $( C^\bullet_{\mathrm{Hom}} (\mathfrak{g}, \mathfrak{g}), [~,~]_{\mathsf{NR}} )$.
\end{thm}

\begin{proof}
(i) Consider the injective map $\Phi := (\theta, \mathrm{id}) : C^\bullet_\mathrm{Hom} (\mathfrak{g}, \mathfrak{g}) \rightarrow C^{\bullet + 1 }_\mathrm{Hom} (\mathfrak{g}, \mathfrak{g}) \oplus C^\bullet_\mathrm{Hom} (\mathfrak{g}, \mathfrak{g})$ of graded vector spaces. It follows from Proposition \ref{ks-prop} that the map
\begin{align*}
\Phi :   ( C^\bullet_\mathrm{Hom} (\mathfrak{g}, \mathfrak{g}), [~,~]_\mathsf{D}) \rightarrow  ( C^{\bullet + 1 }_\mathrm{Hom} (\mathfrak{g}, \mathfrak{g}) \oplus C^\bullet_\mathrm{Hom} (\mathfrak{g}, \mathfrak{g}), [~,~]_\ltimes)
\end{align*}
preserves the corresponding brackets. Since $\Phi$ is injective, it follows that $( C^\bullet_\mathrm{Hom} (\mathfrak{g}, \mathfrak{g}), [~,~]_\mathsf{D})$ is a graded Lie algebra.

(ii) This part follows from the definition of the derived bracket and the identity (\ref{theta-mor}).
\end{proof}

\medskip

\noindent {\bf Rota-Baxter operators on Hom-Lie algebras.} 
Let $(\mathfrak{g}, [~,~], \alpha)$ be a Hom-Lie algebra. Given any fixed $\lambda \in {\bf k}$, we define a map $d_\lambda : C^\bullet_\mathrm{Hom} (\mathfrak{g}, \mathfrak{g}) \rightarrow C^{\bullet +1}_\mathrm{Hom} ( \mathfrak{g}, \mathfrak{g})$ by 
\begin{align*}
d_\lambda (f) := \lambda \delta_\mathrm{Hom}^\mathrm{tr} (f) = - \lambda i_{ \mu} f, \text{ for } f \in C^n_\mathrm{Hom} (\mathfrak{g}, \mathfrak{g}).
\end{align*}
Here $\delta_\mathrm{Hom}^\mathrm{tr}$ is the differential of the Hom-Lie algebra $(\mathfrak{g}, [~,~], \alpha)$ with coefficients in $\mathfrak{g}$ but with the trivial action.  In terms of $\delta_\mathrm{Hom}$ and the map $\theta$, we can write $d_\lambda (f) = \lambda \big( \delta_\mathrm{Hom} f + (-1)^{n-1} \theta f \big)$, for $f \in C^n_\mathrm{Hom} (\mathfrak{g}, \mathfrak{g})$. Explicitly,
\begin{align*}
    (d_\lambda f) (x_1, \ldots, x_{n+1} ) := \lambda \sum_{1 \leq i < j \leq n+1} (-1)^{i+j}  f ( [x_i, x_j] , \alpha (x_i), \ldots, \widehat{ \alpha (x_i)}, \ldots, \widehat{ \alpha (x_j)}, \ldots, \alpha (x_{n+1}) ),
\end{align*}
for $f \in C^n_\mathrm{Hom} (\mathfrak{g}, \mathfrak{g})$ and $x_1, \ldots, x_{n+1} \in \mathfrak{g}$. Then we have the following.

\begin{proposition}\label{prop-dgla}
    The differential map $d_\lambda$ is a graded derivation for the derived bracket. That is, $(C^\bullet_\mathrm{Hom}  (\mathfrak{g}, \mathfrak{g}), [~,~]_\mathsf{D}, d_\lambda )$ is a differential graded Lie algebra.
\end{proposition}

\begin{proof}
The map $d_\lambda$ is a differential as $\delta_\mathrm{Hom}^\mathrm{tr}$ is so. Thus, we only need to show that $d_\lambda$ is a graded derivation for the bracket $[~,~]_\mathsf{D}$. When $\lambda = 0 $, we have $d_\lambda = 0$ and hence it is a graded derivation. Next, we suppose that $\lambda \neq 0$. Since $(d_\lambda)^2 =0$, we have
\begin{align*}
    (\delta_\mathrm{Hom} + (-1)^n \theta) (\delta_\mathrm{Hom} + (-1)^{n-1} \theta) (f) = 0, \text{ for } f \in C^n_\mathrm{Hom} (\mathfrak{g}, \mathfrak{g}).
\end{align*}
This in turn implies that $\big(  (-1)^{n-1} \delta_\mathrm{Hom} \circ \theta + (-1)^n \theta \circ \delta_\mathrm{Hom} - \theta^2 \big)(f) = 0$ or equivalently
\begin{align}\label{theta-2}
    \theta^2 (f) = (-1)^n ( \theta \circ \delta_\mathrm{Hom} - \delta_\mathrm{Hom} \circ \theta) (f), \text{ for } f \in C^n_\mathrm{Hom} (\mathfrak{g}, \mathfrak{g}).
\end{align}
Moreover, since $d_\lambda = \lambda d_{\lambda = 1}$, it is enough to prove the graded derivation property of $d_\lambda$ for $\lambda = 1$. For any $P \in C^m_\mathrm{Hom} (\mathfrak{g}, \mathfrak{g})$ and $Q \in C^n_\mathrm{Hom} (\mathfrak{g}, \mathfrak{g})$, we have
\begin{align}\label{d-lambda-1}
    &d_{\lambda = 1} ( [P, Q]_\mathsf{D}) \nonumber \\
    &= \delta_\mathrm{Hom} ([P, Q]_\mathsf{D}) +(-1)^{m+n-1} \theta ([P, Q]_\mathsf{D}) \nonumber \\
    &= \delta_\mathrm{Hom} \big(  [P, Q]_\mathsf{C} + i_{\theta P} Q - (-1)^{mn} i_{\theta Q} P \big) + (-1)^{m+n-1} [\theta P, \theta Q]_\mathsf{NR} \nonumber \\
    &= [\delta_\mathrm{Hom} P, Q]_\mathsf{C} + (-1)^m [P, \delta_\mathrm{Hom} Q]_\mathsf{C} + \delta_\mathrm{Hom} i_{\theta P} Q - (-1)^{mn} \delta_\mathrm{Hom} i_{\theta Q} P \\
    & \qquad \qquad \qquad \qquad \qquad + (-1)^{m+n-1} \big( i_{\theta P} \theta Q - (-1)^{mn} i_{\theta Q} \theta P   \big). \nonumber
\end{align}
On the other hand,
\begin{align}\label{d-lambda-2}
    &[d_{\lambda = 1} (P), Q]_\mathsf{D} + (-1)^m [P, d_{\lambda =1} (Q)]_\mathsf{D} \nonumber \\
    &= [\delta_\mathrm{Hom} P, Q]_\mathsf{D} + (-1)^{m-1} [\theta P, Q]_\mathsf{D} + (-1)^m [P, \delta_\mathrm{Hom} Q]_\mathsf{D} + (-1)^{m+n-1} [P, \theta Q]_\mathsf{D} \nonumber \\
   & = [ \delta_\mathrm{Hom} P, Q]_\mathsf{C} + i_{\theta \delta_\mathrm{Hom} P} Q - (-1)^{(m+1) n} i_{\theta Q} \delta_\mathrm{Hom} P \\
   & ~~+ (-1)^{m-1} \big\{  i_{\theta P} \delta_\mathrm{Hom} Q + (-1)^m i_{\delta_\mathrm{Hom} \theta P} Q + (-1)^{m+1} \delta_\mathrm{Hom} i_{\theta P} Q + i_{\theta^2 P} Q - (-1)^{(m+1) n} i_{\theta Q} \theta P   \big\} \nonumber \\
   &~~+ (-1)^m \big\{  [P, \delta_\mathrm{Hom} Q]_\mathsf{C} + i_{\theta P} \delta_\mathrm{Hom} Q - (-1)^{m(n+1)} i_{\theta \delta_\mathrm{Hom} Q} P   \big\} 
   + (-1)^{m+n-1} \big\{  - (-1)^{m (n+1)} i_{\theta Q} \delta_\mathrm{Hom} P  \nonumber \\
   & \qquad \qquad - (-1)^{mn + m+n } i_{\delta_\mathrm{Hom} \theta Q} P + (-1)^{mn+m+n} \delta_\mathrm{Hom} (i_{\theta Q} P) + i_{\theta P} \theta Q - (-1)^{m (n+1)} i_{\theta ^2 Q} P \big\}. \nonumber
\end{align}
Comparing expressions in (\ref{d-lambda-1}) and (\ref{d-lambda-2}), we get that
\begin{align*}
    [d_{\lambda = 1} (P), Q]_\mathsf{D} + (-1)^m [P, d_{\lambda =1} (Q)]_\mathsf{D} =~& d_{\lambda = 1} ([P, Q]_\mathsf{D}) + \underbrace{ i_{\theta \delta_\mathrm{Hom} P} Q - i_{\delta_\mathrm{Hom} \theta P} Q - (-1)^m i_{\theta^2 P}Q}  \\
    & \underbrace{ - (-1)^{mn} \big(    i_{\theta \delta_\mathrm{Hom} Q} P - i_{\delta_\mathrm{Hom} \theta Q} P - (-1)^n i_{\theta^2 Q}P  \big)}.
\end{align*}
It follows from (\ref{theta-2}) that both the above underlined expressions are zero. This proves that $d_{\lambda = 1}$ is a graded derivation and hence $d_\lambda = \lambda d_{\lambda =1}$ is so.
\end{proof}

Let $R: \mathfrak{g} \rightarrow \mathfrak{g}$ be a linear map satisfying $\alpha \circ R = R \circ \alpha$. That is, $R \in C^1_\mathrm{Hom} (\mathfrak{g}, \mathfrak{g})$. For any fixed $\lambda \in {\bf k}$, we define a bilinear bracket $[~,~]^R : \mathfrak{g} \times \mathfrak{g} \rightarrow \mathfrak{g}$ by
\begin{align}\label{rota-br}
[x,y]^R := [ Rx, y] + [x, Ry] + \lambda [x, y], ~ \text{ for } x, y \in \mathfrak{g}.
\end{align}
This bracket is clearly skew-symmetric and satisfies $\alpha ([x,y]^R) = [\alpha (x), \alpha (y)]^R$, for $x, y \in \mathfrak{g}$.

\begin{lemma}\label{new-lemma}
 For any $R \in C^1_\mathrm{Hom} (\mathfrak{g}, \mathfrak{g})$, we have   
 \begin{itemize}
 \item[(i)] $[\mu, R]_\mathsf{C} = i_{\theta R} \mu$,
     \item[(ii)] $\theta ( d_\lambda R) = - \lambda [\mu, \theta R ]_\mathsf{NR}$.
 \end{itemize}
\end{lemma}

\begin{proof}
(i) For any $x,y,z \in \mathfrak{g}$,
\begin{align}\label{lemma-mr1}
    [\mu,R ]_\mathsf{C} (x,y,z) = [ [x,y],\alpha (Rz)]-[[x,z],\alpha R(y)]+ [[y,z],\alpha R (x)]
\end{align}
On the other hand, we first have $(\theta R)(x, y)= - [R x, y] - [x, Ry]$, for any $x, y \in \mathfrak{g}$. Hence
\begin{align}
    (i_{\theta R} \mu ) (x, y, z) =~& \mu \big(  (\theta R) (x, y), \alpha (z)  \big)  - \mu \big(  (\theta R) (x, z), \alpha (y)   \big)  + \mu \big( (\theta R) (y, z), \alpha (x) \big) \nonumber \\
    =~& - [[ Rx, y], \alpha (z)] - [[x, R(y)], \alpha (z) ] + [[ R(x), z], \alpha (y)] + [[ x, R(z)], \alpha (y)] \label{lemma-mr2} \\
    & \qquad \qquad - [ [Ry, z], \alpha (x)] - [ [ y, Rz], \alpha (x)]. \nonumber
\end{align}
The expressions in (\ref{lemma-mr1}) and (\ref{lemma-mr2}) are the same by using the Hom-Jacobi identity.

(ii) For this part, we observe that 
\begin{align*}
    \theta (d_\lambda R) + \lambda [\mu, \theta R]_\mathsf{NR} =~& - \lambda \theta (i_\mu R) + \lambda (i_\mu \theta R + i_{\theta R} \mu) \\
    =~& \lambda \big(  i_\mu \theta R - \theta (i_\mu R)  \big) + \lambda (i_{\theta R} \mu) \\
    =~& - \lambda [\mu, R]_\mathsf{C} + \lambda (i_{\theta R} \mu) \quad (\text{by Lemma } \ref{lemma-lemma})\\
    =~& 0 \quad (\text{by part }(\mathrm{i})).
\end{align*}
Hence the result follows.
\end{proof}

In the next result, we give a necessary and sufficient condition under which the bracket (\ref{rota-br}) defines a new Hom-Lie algebra structure on $\mathfrak{g}$.

\begin{proposition}\label{prop-dgla-rota}
Let $(\mathfrak{g}, [~,~], \alpha)$ be a Hom-Lie algebra and $R : \mathfrak{g} \rightarrow \mathfrak{g}$ be a linear map satisfying $\alpha \circ R = R \circ \alpha$.
Then the triple $(\mathfrak{g}, [~,~]^R, \alpha)$ is a Hom-Lie algebra if and only if 
\begin{align*}
\theta (d_\lambda R + \frac{1}{2} [R,R]_\mathsf{D} ) =0.
\end{align*}
\end{proposition}

\begin{proof}
For any $x, y \in \mathfrak{g}$, we first observe that
\begin{align*}
( \lambda \mu - \theta R ) (x, y) =  [Rx, y] + [x, Ry] + \lambda [x, y]=  [x, y]^R.
\end{align*}
Hence $[~,~]^R$ is a Hom-Lie bracket if and only if $[ \lambda \mu - \theta R, \lambda \mu - \theta R]_\mathsf{NR} = 0$. This is same as 
\begin{align*}
    - 2 \lambda [\mu, \theta R]_\mathsf{NR} + [\theta R, \theta R]_\mathsf{NR} = 0 \qquad \text{ equivalently, } \quad \theta \big( d_\lambda R + \frac{1}{2} [R,R]_\mathsf{D} \big) = 0.
\end{align*}
The last equivalence follows from Lemma \ref{new-lemma} (ii) and the fact that 
$\theta ([R, R]_\mathsf{D}) = [\theta R, \theta R]_\mathsf{NR}.$ This proves the result.
\end{proof}

\begin{definition}
Let $(\mathfrak{g},[~,~], \alpha)$ be a Hom-Lie algebra. A linear map $R : \mathfrak{g} \rightarrow \mathfrak{g}$ satisfying $\alpha \circ R = R \circ \alpha$ is said to be a {\bf Rota-Baxter operator of weight $\lambda \in {\bf k}$} on the Hom-Lie algebra if
\begin{align}\label{rb-identity}
[Rx, Ry] = R ( [Rx, y ] + [x, Ry] + \lambda [x, y]),~ \text{ for } x, y \in \mathfrak{g}.
\end{align}
\end{definition}

Note that the condition (\ref{rb-identity}) is equivalent to say that 
\begin{align*}
d_\lambda (R) + \frac{1}{2}[R, R]_\mathsf{D} = 0,
\end{align*}
i.e. $R$ is a Maurer-Cartan element in the differential graded Lie algebra $\big( C^\bullet_\mathrm{Hom} (\mathfrak{g}, \mathfrak{g}) , [~,~]_\mathsf{D}, d_\lambda \big)$. In particular, $R$ is a Rota-Baxter operator of weight $0$ if and only if it is a Maurer-Cartan element in the graded Lie algebra $\big( C^\bullet_\mathrm{Hom} (\mathfrak{g}, \mathfrak{g}) , [~,~]_\mathsf{D})$.

\begin{proposition}
Let $R$ be a Rota-Baxter operator on a Hom-Lie algebra $(\mathfrak{g}, [~,~], \alpha)$. Then $(\mathfrak{g}, [~,~]^R, \alpha)$ is a Hom-Lie algebra and $R : (\mathfrak{g}, [~,~]^R, \alpha) \rightarrow (\mathfrak{g}, [~,~], \alpha)$ is a morphism of Hom-Lie algebras.
\end{proposition}

\begin{proof}
The first part follows from the previous proposition and the second part follows from the identity (\ref{rb-identity}).
\end{proof}

\medskip

\noindent {\bf Relative Rota-Baxter operators on Hom-Lie algebras.} Here we will generalize the graded Lie algebra and the differential graded Lie algebra constructed above to study relative Rota-Baxter operators.

Let $(\mathfrak{g}, [~,~]_\mathfrak{g}, \alpha)$ and $(\mathfrak{h}, [~,~]_\mathfrak{h}, \beta)$ be two Hom-Lie algebras. A {\em Hom-Lie algebra action} of $(\mathfrak{g}, [~,~]_\mathfrak{g}, \alpha)$ on $(\mathfrak{h}, [~,~]_\mathfrak{h}, \beta)$ is a bilinear map $\diamond : \mathfrak{g} \times \mathfrak{h} \rightarrow \mathfrak{h}$ that makes the triple $(\mathfrak{h}, \diamond, \beta)$ into a representation of the Hom-Lie algebra $(\mathfrak{g}, [~,~]_\mathfrak{g}, \alpha)$, i.e.
\begin{align*}
    \beta (x \diamond h) = \alpha (x) \diamond \beta (h) ~~~~ \text{ and } ~~~~ [x, y]_\mathfrak{g} \diamond \beta (h) = \alpha (x) \diamond (y \diamond h) - \alpha (y) \diamond (x \diamond h), \text{ for } x, y \in \mathfrak{g}, h \in \mathfrak{h},
\end{align*}
satisfying additionally
\begin{align*}
  \alpha (x) \diamond [h, k]_\mathfrak{h} = [ x \diamond h , \beta (k)]_\mathfrak{h} + [ \beta (h), x \diamond k]_\mathfrak{h}, \text{ for } x \in \mathfrak{g} \text{ and } h, k \in \mathfrak{h}.  
\end{align*}

It follows that any Hom-Lie algebra $(\mathfrak{g}, [~,~]_\mathfrak{g}, \alpha)$ has a Hom-Lie algebra action on itself with the action map $\diamond$ being given by the Hom-Lie bracket $[~,~]_\mathfrak{g}$. This is called the adjoint action. The proof of the following result is straightforward.

\begin{proposition}
    Let $(\mathfrak{g}, [~,~]_\mathfrak{g}, \alpha)$ and $(\mathfrak{h}, [~,~]_\mathfrak{h}, \beta)$ be two Hom-Lie algebras and $\diamond : \mathfrak{g} \times \mathfrak{h} \rightarrow \mathfrak{h}$ be a Hom-Lie algebra action. Then for any $\lambda \in {\bf k}$, the direct sum $\mathfrak{g} \oplus \mathfrak{h}$ carries a bilinear skew-symmetric bracket
    \begin{align*}
        [(x, h), (y, k)]_\ltimes^\lambda := ([x, y]_\mathfrak{g}, x \diamond k - y \diamond h + \lambda [h, k]_\mathfrak{h}), \text{ for } (x, h), (y, k) \in \mathfrak{g} \oplus \mathfrak{h}
    \end{align*}
    that makes $( \mathfrak{g} \oplus \mathfrak{h}, [~,~]_\ltimes^\lambda, \alpha \oplus \beta)$ into a Hom-Lie algebra, called the semidirect product of weight $\lambda$.
\end{proposition}

Let $(\mathfrak{g}, [~,~]_\mathfrak{g}, \alpha)$ and $(\mathfrak{h}, [~,~]_\mathfrak{h}, \beta)$ be two Hom-Lie algebras and $\diamond : \mathfrak{g} \times \mathfrak{h} \rightarrow \mathfrak{h}$ be a Hom-Lie algebra action. A linear map $R : \mathfrak{h} \rightarrow \mathfrak{g}$ is said to be a {\bf relative Rota-Baxter operator of weight $\lambda$} with respect to the Hom-Lie algebra action if $\alpha \circ R = R \circ \beta $ and
\begin{align*}
    [R(h), R(k)]_\mathfrak{g} = R \big( R(h) \diamond k - R(k) \diamond h + \lambda [h, k]_\mathfrak{h} \big), \text{ for } h, k \in \mathfrak{h}.
\end{align*}

It follows that a Rota-Baxter operator of weight $\lambda$ on a Hom-Lie algebra can be regarded as a relative Rota-Baxter operator of weight $\lambda$ with respect to the adjoint action. We have the following simple characterization of a relative Rota-Baxter operator of weight $\lambda$.

\begin{proposition}
    With the assumptions of the above proposition, a linear map $R: \mathfrak{h} \rightarrow \mathfrak{g}$ is a relative Rota-Baxter operator of weight $\lambda$ if and only if its graph $Gr (R) = \{ (R(h), h) | h \in \mathfrak{h} \}$ is a Hom-Lie subalgebra of $( \mathfrak{g} \oplus \mathfrak{h}, [~,~]_\ltimes^\lambda, \alpha \oplus \beta)$.
\end{proposition}

Since $Gr (R)$ is linearly isomorphic to the vector space $\mathfrak{h}$, the above proposition shows the existence of an induced Hom-Lie algebra structure on $\mathfrak{h}$. This induced Hom-Lie algebra structure is precisely given by $(\mathfrak{h}, [~,~]_\mathfrak{h}^R, \beta)$, where
\begin{align}\label{hom-induced}
    [h, k]_\mathfrak{h}^R := R(h) \diamond k - R(k) \diamond h + \lambda [h, k]_\mathfrak{h}, \text{ for } h, k \in \mathfrak{h}.
\end{align}

In the following, we construct a graded Lie algebra (resp. a differential graded Lie algebra) whose Maurer-Cartan elements are precisely relative Rota-Baxter operators of weight $0$ (resp. of nonzero weight $\lambda$). As before, we let $(\mathfrak{g}, [~,~]_\mathfrak{g}, \alpha)$ and $(\mathfrak{h}, [~,~]_\mathfrak{h}, \beta)$ be two Hom-Lie algebras and $\diamond : \mathfrak{g} \times \mathfrak{h} \rightarrow \mathfrak{h}$ be a Hom-Lie algebra action. We set 
\begin{align*}
C^n_\mathrm{Hom} (\mathfrak{h}, \mathfrak{g}) = \{ f : \wedge^n \mathfrak{h} \rightarrow \mathfrak{g} ~|~ \alpha \circ f = f \circ \beta^{\wedge^n} \}
\end{align*}
and take $C^\bullet_\mathrm{Hom} (\mathfrak{h}, \mathfrak{g}) = \bigoplus_n C^n_\mathrm{Hom} (\mathfrak{h}, \mathfrak{g}).$ Generalizing the contraction map (\ref{contraction-map}) and the map (\ref{theta}), we define maps 
\begin{align*}
    \widetilde{i} : C^m_\mathrm{Hom} (\mathfrak{h}, \mathfrak{h}) \times C^n_\mathrm{Hom} (\mathfrak{h}, \mathfrak{g}) \rightarrow C^{m+n-1}_\mathrm{Hom} (\mathfrak{h}, \mathfrak{g}) \quad \text{ and } \quad \widetilde{\theta} : C^n_\mathrm{Hom} (\mathfrak{h}, \mathfrak{g}) \rightarrow C^{n+1}_\mathrm{Hom} (\mathfrak{h}, \mathfrak{h})
\end{align*}
by
\begin{align*}
    (\widetilde{i}_f P) (h_1, \ldots, h_{m+n-1}) =~&  \sum_{\sigma \in Sh (m, n-1)} (-1)^\sigma P \big(  f (h_{\sigma (1)}, \ldots, h_{\sigma (m)} ), \beta^{m-1} h_{\sigma (m+1)}, \ldots, \beta^{m-1} h_{\sigma (m+n-1)}  \big),\\
    (\widetilde{\theta} P) (h_1, \ldots, h_{n+1})=~& \sum_{i=1}^{n+1} (-1)^{n+i}~ P (h_1, \ldots, \widehat{h_i}, \ldots, h_{n+1}) \diamond \beta^{n-1} (h_i),
\end{align*}
for $f \in C^m_\mathrm{Hom} (\mathfrak{h}, \mathfrak{h})$, $P \in C^n_\mathrm{Hom} (\mathfrak{h}, \mathfrak{g})$ and $h_1, \ldots, h_{m+n-1} \in \mathfrak{h}$. Note that, we also have the cup product bracket $[~,~]_\mathsf{C} : C^m_\mathrm{Hom} (\mathfrak{h}, \mathfrak{g}) \times C^n_\mathrm{Hom} (\mathfrak{h}, \mathfrak{g}) \rightarrow C^{m+n}_\mathrm{Hom} (\mathfrak{h}, \mathfrak{g})$ given by (\ref{cup-form}). With the above notations, we define a bracket (also called the {\bf derived bracket}) $[~,~]_\mathsf{D}^{\sim}: C^m_\mathrm{Hom} (\mathfrak{h}, \mathfrak{g}) \times C^n_\mathrm{Hom} (\mathfrak{h}, \mathfrak{g}) \rightarrow C^{m+n}_\mathrm{Hom} (\mathfrak{h}, \mathfrak{g}) $ by
\begin{align*}
    [P, Q]_\mathsf{D}^\sim := [P, Q]_\mathsf{C} + \widetilde{i}_{ \widetilde{\theta} P} Q - (-1)^{mn}~ \widetilde{i}_{ \widetilde{\theta} Q} P,
\end{align*}
for $P \in C^m_\mathrm{Hom} (\mathfrak{h}, \mathfrak{g})$ and $Q \in C^n_\mathrm{Hom} (\mathfrak{h}, \mathfrak{g})$. Explicitly, we have
\begin{align*}
     &[P, Q]_\mathsf{D}^\sim (h_1, \ldots, h_{m+n}) \\
     &~= \sum_{\sigma \in Sh (m,n)} (-1)^\sigma [\alpha^{n-1}P(h_{\sigma (1)}, \ldots , h_{\sigma (m)}), \alpha^{m-1} Q(h_{\sigma (m+1)},\ldots , _{\sigma (m+n)})]_\mathfrak{g} \\ 
    &~~- \sum_{\sigma \in Sh (m,1,n-1)} (-1)^\sigma Q\big( P(h_{\sigma (1)}, \ldots , h_{\sigma (m)}) \diamond \beta^{m-1}h_{\sigma (m+1)} , \beta^m h_{\sigma (m+2)},\ldots , \beta^m h_{\sigma (m+n)}  \big) \\ 
    &~~+ (-1)^{mn} \sum_{\sigma \in Sh (n,1,m-1)} (-1)^\sigma P\big( Q(h_{\sigma (1)}, \ldots , h_{\sigma (n)}) \diamond \beta^{n-1}h_{\sigma (n+1)}, \beta^n h_{\sigma (n+2)},\ldots , \beta^n h_{\sigma (m+n)}  \big),
\end{align*}
for $h_1, \ldots, h_{m+n} \in \mathfrak{h}$. Observe that, to define the bracket $[~,~]_\mathsf{D}^\sim$, we don't require the Hom-Lie bracket of $\mathfrak{h}$. Therefore, the bracket $[~,~]_\mathsf{D}^\sim$ is still defined when $(\mathfrak{h}, \diamond , \beta)$ is only a representation of the Hom-Lie algebra $(\mathfrak{g}, [~,~]_\mathfrak{g}, \alpha)$. By using the Hom-Lie bracket of $\mathfrak{h}$, we now define a map $\widetilde{d}_\lambda : C^n_\mathrm{Hom} (\mathfrak{h}, \mathfrak{g}) \rightarrow C^{n+1}_\mathrm{Hom} (\mathfrak{h}, \mathfrak{g})$ by
\begin{align*}
   ( \widetilde{d}_\lambda f) (h_1, \ldots, h_{n+1}) := \lambda \sum_{1 \leq i < j \leq n+1} (-1)^{i+j} f \big(  [h_i, h_j]_\mathfrak{h}, \beta (h_1), \ldots, \widehat{ \beta (h_i)}, \ldots, \widehat{ \beta (h_j)} , \ldots, \beta (h_{n+1})  \big),
\end{align*}
for $f \in C^n_\mathrm{Hom} (\mathfrak{h}, \mathfrak{g})$ and $h_1, \ldots, h_{n+1} \in \mathfrak{h}$. Then we have $(\widetilde{d}_\lambda)^2 = 0$. Moreover, we have the following result (as a generalization of Propositions \ref{prop-dgla} and \ref{prop-dgla-rota}).

\begin{thm}
    Let $(\mathfrak{g}, [~,~]_\mathfrak{g}, \alpha)$ and $(\mathfrak{h}, [~,~]_\mathfrak{h}, \beta)$ be two Hom-Lie algebras and $\diamond : \mathfrak{g} \times \mathfrak{h} \rightarrow \mathfrak{h}$ be a Hom-Lie algebra action.

    (i) Then for any $\lambda \in {\bf k}$, the triple $\big(  C^\bullet_\mathrm{Hom} (\mathfrak{h}, \mathfrak{g}), [~, ~]_\mathsf{D}^\sim,  \widetilde{d}_\lambda \big)$ is a differential graded Lie algebra.

    (ii) An element $R \in C^1_\mathrm{Hom} (\mathfrak{h}, \mathfrak{g})$ is a Maurer-Cartan element of the above differential graded Lie algebra if and only if $R: \mathfrak{h} \rightarrow \mathfrak{g}$ is a relative Rota-Baxter operator of weight $\lambda$ with respect to the Hom-Lie algebra action $\diamond$.
\end{thm}

The above theorem shows that a relative Rota-Baxter operator of weight $\lambda$ can be seen as a Maurer-Cartan element in a suitable differential graded Lie algebra. Using this characterization, one can define the cohomology (and deformations) of a relative Rota-Baxter operator of weight $\lambda$. Let $R: \mathfrak{h} \rightarrow \mathfrak{g}$ be a relative Rota-Baxter operator of weight $\lambda$ with respect to a given Hom-Lie algebra action $\diamond: \mathfrak{g} \times \mathfrak{h} \rightarrow \mathfrak{h}$. We define a map $D_R : C^n_\mathrm{Hom} (\mathfrak{h}, \mathfrak{g}) \rightarrow  C^{n+1}_\mathrm{Hom} (\mathfrak{h}, \mathfrak{g})$, for $n \geq 0$, by
\begin{align*}
    D_R (f) = \widetilde{d}_\lambda (f) + [R, f]_\mathsf{D}^\sim, \text{ for } f \in C^n_\mathrm{Hom} (\mathfrak{h}, \mathfrak{g}).
\end{align*}
Then we have $(D_R)^2 = 0$ and hence $\{ C^\bullet_\mathrm{Hom} (\mathfrak{h}, \mathfrak{g}), D_R \}$ is a cochain complex. The corresponding cohomology groups are called the {\em cohomology groups} of the operator $R$.

\begin{remark}
    When $\mathfrak{g}$ and $\mathfrak{h}$ are both Lie algebras (i.e. $\alpha = \mathrm{id}_\mathfrak{g}$ and $\beta = \mathrm{id}_\mathfrak{h}$), one recover the cohomology of a relative Rota-Baxter operator of weight $\lambda$ with respect to a Lie algebra action as defined in \cite{das-weighted}.
\end{remark}

In the following, we observe that the cohomology of a relative Rota-Baxter operator of weight $\lambda$ can be viewed as the cohomology of the induced Hom-Lie algebra with coefficients in a suitable representation. When $\lambda = 0$, a similar result has been proved in \cite{mishra-naolekar}.

\begin{proposition}
    Let $R : \mathfrak{h} \rightarrow \mathfrak{g}$ be a relative Rota-Baxter operator of weight $\lambda$ with respect to a given Hom-Lie algebra action $\diamond : \mathfrak{g} \times \mathfrak{h} \rightarrow \mathfrak{h}$. We define a bilinear map
    \begin{align*}
        \widetilde{\diamond} : \mathfrak{h} \times \mathfrak{g} \rightarrow \mathfrak{g} ~~ \text{ by }~ h ~ \! \widetilde{ \diamond }~ \! x := [R(h), x]_\mathfrak{g} + R (x \diamond h), \text{ for } h \in \mathfrak{h} \text{ and }  x \in \mathfrak{g}.
    \end{align*}
    Then $(\mathfrak{g}, \widetilde{\diamond}, \alpha)$ is a representation of the induced Hom-Lie algebra $(\mathfrak{h}, [~, ~]_\mathfrak{h}^R, \beta)$. Moreover, the corresponding cohomology groups of the induced Hom-Lie algebra are isomorphic to the cohomology groups of the operator $R$.
\end{proposition}

\begin{proof}
    The first part is a generalization of \cite[Proposition 5.8]{das-weighted} in the Hom-Lie algebra context. Hence we will not repeat it here.

    For the second part, we let $\delta_\mathrm{Hom}^R : C^n_\mathrm{Hom} (\mathfrak{h}, \mathfrak{g}) \rightarrow  C^{n+1}_\mathrm{Hom} (\mathfrak{h}, \mathfrak{g})$ be the differential of the induced Hom-Lie algebra $(\mathfrak{h}, [~,~]_\mathfrak{h}^R, \beta)$ with coefficients in the representation $(\mathfrak{g}, \widetilde{ \diamond }, \alpha)$. Then we have
    \begin{align*}
        (\delta_\mathrm{Hom}^R f) (h_1, \ldots, h_{n+1}) =~& \sum_{i=1}^{n+1} (-1)^{i+1}~ \beta^{n-1} (h_i) ~\! \widetilde{ \diamond } ~ \! f(h_1, \ldots, \widehat{h_i}, \ldots, h_{n+1})\\
       & + \sum_{1 \leq i < j \leq n+1} (-1)^{i+j} ~ \! f ([h_i, h_j]_\mathfrak{h}^R, \beta (h_1), \ldots,\widehat{ \beta (h_i)}, \ldots, \widehat{ \beta (h_j)}, \ldots, \beta (h_{n+1}))\\
        =~& \widetilde{d}_\lambda (f) + [R, f]_\mathsf{D}^\sim = D_R (f),
    \end{align*}
for $f \in C^n_\mathrm{Hom} (\mathfrak{h}, \mathfrak{g})$ and $h_1, \ldots, h_{n+1} \in \mathfrak{h}$. Hence the result follows.
\end{proof}



\noindent {\bf Concluding remark.} For the Jackson $\mathfrak{sl}_2$ given in Example \ref{jack}, we observe that 
\begin{align*}
    \alpha ([e, f]) = \frac{1+q}{2} \alpha (h) = \frac{q (1+q)}{2} h ~~~ \text{ and } ~~~~ [\alpha (e) ,\alpha (f) ] = q^3 [e, f] = \frac{q^3 (1+q)}{2} h.
\end{align*}
Hence $\alpha ([e, f])  \neq [\alpha (e),\alpha (f)]$ in general which shows that the Jackson $\mathfrak{sl}_2$ is not a multiplicative Hom-Lie algebra. Thus, all the constructions made in this paper do not apply to this particular example. In future, we aim to construct the cup product, Fr\"{o}licher-Nijenhuis bracket and the derived bracket associated with such non-multiplicative Hom-Lie algebras.

\medskip

    \noindent {\bf Acknowledgements.} Anusuiya Baishya would like to acknowledge the financial support received from the Department of Science and Technology (DST), Government of India through INSPIRE Fellowship [IF210619]. Both authors thank the Department of Mathematics, IIT Kharagpur for providing the beautiful academic atmosphere where the research has been carried out.
    
\medskip

\noindent {\bf Conflict of interest statement.} On behalf of all authors, the corresponding author states that there is no conflict of interest.

\medskip

\noindent {\bf Data Availability Statement.} Data sharing does not apply to this article as no new data were created or analyzed in this study.

\end{document}